\documentclass[11pt]{amsart}
\usepackage[margin=1.1in]{geometry}
\usepackage{amsmath,amssymb,amsthm,mathtools,booktabs,enumitem}
\usepackage[hidelinks]{hyperref}

\newtheorem{theorem}{Theorem}[section]
\newtheorem{proposition}[theorem]{Proposition}
\newtheorem{lemma}[theorem]{Lemma}

\newtheorem{conjecture}[theorem]{Conjecture}
\theoremstyle{remark}
\newtheorem{remark}[theorem]{Remark}
\theoremstyle{definition}

\newcommand{\Z}{\mathbb Z}
\newcommand{\Q}{\mathbb Q}

\newcommand{\HH}{\mathbb H}
\newcommand{\cG}{\mathcal G}
\newcommand{\cP}{\mathcal P}
\newcommand{\cE}{\mathcal E}
\newcommand{\cO}{\mathcal O}
\newcommand{\cR}{\mathcal R}
\newcommand{\sgn}{\operatorname{sgn}}
\newcommand{\Rea}{\operatorname{Re}}
\newcommand{\Ima}{\operatorname{Im}}
\newcommand{\Log}{\operatorname{Log}}
\newcommand{\Arg}{\operatorname{Arg}}
\newcommand{\e}{\mathrm e}
\newcommand{\KNG}{K_{\mathrm{NG}}}

\title[Sign patterns of real powers of infinite products]{Sign patterns of real powers of infinite products: resolution of four conjectures of Schlosser and Zhou}
\author{Jaideep Sai Padhi}
\address{Department of Computer Science, Purdue University, West Lafayette, IN, USA}
\email{jpadhi@purdue.edu}
\date{\today}

\begin{document}
\emergencystretch=2em
\begin{abstract}
For an infinite product $P(q)=\prod_{m\ge1}(1-q^m)^{\epsilon(m)}$ and real $\delta$, write $P(q)^\delta=\sum_{n\ge0}c_\delta(n)q^n$. Schlosser and Zhou conjectured precise sign patterns for these coefficients, for several products and ranges of $\delta$. We resolve four of their conjectures completely, in each case extending partial results already in the literature: those for the G\"ollnitz--Gordon product $Q_8=(q,q^7;q^8)_\infty/(q^3,q^5;q^8)_\infty$ (Conjecture 21), for $Q_{12}=(q,q^{11};q^{12})_\infty/(q^5,q^7;q^{12})_\infty$ (Conjecture 24), and for the Borwein products $G_7$ and $G_{11}$ (Conjectures 20 and 23).
\begin{itemize}[leftmargin=1.5em]
\item Conjecture 23 is true.
\item Conjecture 21 fails exactly on $[\beta,8/3)$, where $\beta\approx2.66448$. The threshold $8/3$ is sharp, and the first counterexamples occur near $n\approx7.6\cdot10^5$.
\item Conjecture 24 fails exactly on $(\delta_1,0)$, where $\delta_1=-0.64411\ldots$ is a root of the $22$nd coefficient polynomial.
\item Conjecture 20 fails exactly on $(\delta_c,5)$, where $\delta_c=4.8735075853867342634\ldots$ is a root of the $897$th coefficient polynomial.
\end{itemize}
All other stated ranges are proved. Parts of Conjecture 21 were settled independently by He and Li; what is new here is the range near $\delta=-1$, where their threshold diverges. Each failure has the same mechanism: a leading circle-method amplitude vanishes, and a secondary term with the wrong sign overtakes it. The proofs are self-contained apart from classical facts. The cusp analysis is exact, via finite orbits of Weil representations; the Hardy--Ramanujan--Rademacher expansion is made fully explicit; and certified computations in exact and ball arithmetic handle the finite ranges and the degenerate regimes.
\end{abstract}

\maketitle

\section{Introduction}
Throughout, a product $P(q)=\prod_{m\ge1}(1-q^m)^{\epsilon(m)}$ with $\epsilon$ periodic is raised to a real power, and
\begin{equation}\label{eq:rec}
P(q)^\delta=\exp(\delta\log P(q))=\sum_{n\ge0}c_\delta(n)q^n,\qquad n\,c_\delta(n)=\delta\sum_{k=1}^nb_k\,c_\delta(n-k),\qquad b_k=-\sum_{d\mid k}\epsilon(d)\,d .
\end{equation}
In particular each $c_\delta(n)\in\Q[\delta]$ has degree at most $n$. Schlosser and Zhou \cite[Appendix A]{SZ} conjectured precise sign patterns for the following products.

\begin{conjecture}[{\cite[Conj.~21]{SZ}}]\label{conj}
For $Q_8=(q,q^7;q^8)_\infty/(q^3,q^5;q^8)_\infty$ the coefficients exhibit
\begin{enumerate}[label=(\alph*),leftmargin=2em]
\item for $\delta=2$ the length $16$ pattern $+-++-+--+--+-++-$;
\item for $2.664479110226972\ldots\approx\beta\le\delta\le4$ the pattern
\begin{equation}\label{eq:sigma}
\sigma=(+,-,+,+,-,+,-,-)\qquad(n\equiv0,\dots,7\bmod 8),
\end{equation}
where $\beta$ is the real root in $(2,3)$ of $x^{12}-90x^{11}+1457x^{10}+30486x^9-537081x^8+1892346x^7-3683653x^6-837509646x^5+774767020x^4+3333687384x^3-40887173664x^2+94379731200x+49816166400$;
\item for $-1<\delta\le\frac{7-\sqrt{73}}2$ the pattern $\sigma'=(+,+,+,-,-,-,-,+)$;
\item for $\delta=-2$ the length $16$ pattern $+++++----+++----$.
\end{enumerate}
\end{conjecture}

\begin{conjecture}[{\cite[Conj.~20]{SZ}}]\label{conj20}
For $G_7=(q;q)_\infty/(q^7;q^7)_\infty$ the coefficients exhibit the pattern
\begin{itemize}[leftmargin=2em]
\item $+--00+0$ for $\delta=1$;
\item $+--+++-$ for $2\le\delta<3$;
\item $+-0+00-$ for $\delta=3$;
\item $+-++---$ for $3<\delta\le5$.
\end{itemize}
\end{conjecture}

\begin{conjecture}[{\cite[Conj.~23]{SZ}}]\label{conj23}
For $G_{11}$ the coefficients exhibit the pattern
\begin{itemize}[leftmargin=2em]
\item $+--0-+0+000$ for $\delta=1$;
\item $+--+++-+--+$ for $\gamma\le\delta\le2$, where $\gamma=1.7584535519419\ldots$ is the root in $(1.5,2)$ of $c_\delta(21)$;
\item $+-0+-0-000+$ for $\delta=3$.
\end{itemize}
\end{conjecture}

\begin{conjecture}[{\cite[Conj.~24]{SZ}}]\label{conj24}
For $Q_{12}=(q,q^{11};q^{12})_\infty/(q^5,q^7;q^{12})_\infty$ the coefficients exhibit the pattern
\begin{itemize}[leftmargin=2em]
\item $+-+0-+-+-0+-$ for $\delta=1$;
\item $+-+--+-+-++-$ for $2\le\delta\le3$;
\item $+++++0-----0$ for $\delta=-1$;
\item $+++++------+$ for $-1<\delta<0$.
\end{itemize}
\end{conjecture}

As in \cite{SZ}, a sign pattern allows zero coefficients. The pattern strings above are quoted from the published version of \cite{SZ} (Ramanujan J.\ \textbf{61} (2023), 515--543, Appendix~1); they are partly illegible in the arXiv rendering. They agree with every coefficient we computed, and the forced zeros are explained in \S\ref{sec:zeros}.

\subsection*{Results}
\begin{theorem}\label{thm:A}
For every $\delta\in[8/3,4]$ and every $n\ge0$, $c_\delta(n)$ has the sign $\sigma(n\bmod 8)$ for $Q_8$. The only vanishing coefficient is $c_4(3)=0$.
\end{theorem}

\begin{theorem}\label{thm:B}
Let $2<\delta<8/3$. For $Q_8$ there is $n_1(\delta)$ such that $c_\delta(n)<0$ for all $n\ge n_1(\delta)$ with $n\equiv10\pmod{16}$, and $c_\delta(n)>0$ for all such $n\equiv14\pmod{16}$. For example, $c_{2.66448}(1\,000\,010)<-5.7\cdot10^{387}$. Hence Conjecture~\ref{conj}(b) fails on $[\beta,8/3)$.
\end{theorem}

\begin{theorem}\label{thm:C}
Parts (a), (c) and (d) of Conjecture~\ref{conj} hold.
\end{theorem}

Parts (a) and (d) were obtained independently by He and Li \cite[Thm.~1.4]{HL}, and their
method gives part (c) on $-0.99<\delta\le(7-\sqrt{73})/2$. The new content of
Theorem~\ref{thm:C} is therefore the closed endpoint: part (c) on $-1<\delta\le-0.99$, where the
leading amplitude degenerates as $\delta\to-1$ and the threshold of \cite{HL} diverges. Our proof
covers the whole of $(-1,(7-\sqrt{73})/2]$ uniformly; see \S\ref{sec:rest}.

\begin{theorem}\label{thm:20}
Let $\delta_c=4.87350758538673426336\ldots$ be the unique root of $c_\delta(897)$ for $G_7$ in an interval of width $1.3\cdot10^{-19}$. Conjecture~\ref{conj20} holds for $\delta=1$, for $2\le\delta\le3$, for $3<\delta\le\delta_c$, and for $\delta=5$. It fails for every $\delta\in(\delta_c,5)$.
\end{theorem}

\begin{theorem}\label{thm:23}
Conjecture~\ref{conj23} holds.
\end{theorem}

\begin{theorem}\label{thm:24}
Conjecture~\ref{conj24} holds for $\delta=\pm1$ and for $2\le\delta\le3$. For $-1<\delta<0$ the pattern holds if and only if $\delta\le\delta_1$, where $\delta_1=-0.644113480813476518842928\ldots$ is the unique root of $c_\delta(22)$ in $(-1,0)$. For instance, $c_{-1/4}(6)=545/65536>0$.
\end{theorem}

Conjectures 17--19 and 22 of \cite{SZ} were settled by He and Li \cite{HL} and by Bringmann, Heim and Kane \cite{BHK}. Parts of the four conjectures treated here were also known. As noted in \cite{HL}, the asymptotic formula \cite[Thm.~3]{SZ} already settles Conjectures 20 and 23 partially. That formula is stated for $\delta\in(0,24/(p-1)]$, so for $G_7$ it reaches $\delta\le4$ and for $G_{11}$ only $\delta\le12/5$; the sign consequences were worked out in \cite{SZ} for $p=3$ alone (their Cor.~5). In particular $\delta=3$ in Conjecture~\ref{conj23}, and the interval $(4,5]$ in Conjecture~\ref{conj20} --- which contains the failure point $\delta_c$ of Theorem~\ref{thm:20} --- lie outside the range of \cite[Thm.~3]{SZ}. He and Li \cite{HL} prove sign patterns for $Q_8^\delta$ and $Q_{12}^\delta$ on ranges of $\delta$ disjoint from ours: for $Q_8$ they treat $\delta=2$, $\delta=-2$ and $-0.99<\delta\le(7-\sqrt{73})/2$, and state that their hypothesis fails for $\beta\le\delta\le4$, which is where our Theorem~\ref{thm:B} locates the failure; for $Q_{12}$ they treat $\delta=\pm1$, $\delta=-\tfrac12$ and negative ranges, and cannot reach $2\le\delta\le3$. Their statements on $-0.999<\delta<-0.501$ and $-0.499\le\delta<-0.001$ are asymptotic, valid for $n\ge1277$ and $n\ge1283$ respectively, and so do not meet the small-$n$ failures of Theorem~\ref{thm:24}, which occur at $n=6,15,16,22$. With the present results, every conjecture of \cite[Appendix A]{SZ} is resolved. For integer exponents, Wang \cite{W} characterized the signs of $G_t^m$ completely when $m(t-1)\le24$; our certificates for those cases are independent of \cite{W}.

\subsection*{The mechanism of failure}
In all three failures a leading amplitude of the circle method vanishes, and a secondary term with the wrong sign takes over.
\begin{itemize}[leftmargin=2em]
\item \emph{$Q_8$.} At $n\equiv2,6\pmod8$ the leading term vanishes identically. There the sign is decided by a correction term at the cusps $3/8,5/8$ and by the leading term at the cusps of denominator $16$. Their growth rates coincide exactly at $\delta=8/3$, and below $8/3$ the second, which has the wrong sign at $n\equiv10,14\pmod{16}$, wins.
\item \emph{$G_7$.} At $n\equiv1\pmod7$ the leading amplitude vanishes at $\delta^*=4.87350758538675508\ldots$. Just below $\delta^*$, the leading term at denominator $14$ (with the wrong sign at $n\equiv1\pmod{14}$) wins on a window of $n$, and that window first becomes nonempty at $n=897$.
\item \emph{$Q_{12}$.} The failure on $(\delta_1,0)$ is a small-coefficient phenomenon: $c_\delta(22)$ changes sign at $\delta_1$ and keeps the wrong sign on $(\delta_1,0)$.
\end{itemize}

\subsection*{Organisation}
\begin{itemize}[leftmargin=2em]
\item \S\S\ref{sec:prelim}--\ref{sec:circle} develop the method for $Q_8$: theta quotients, the Weil orbit and its cusp types, and the explicit circle method.
\item \S\S\ref{sec:A}--\ref{sec:rest} prove Theorems~\ref{thm:A}--\ref{thm:C}.
\item \S\ref{sec:zeros} explains the forced zeros.
\item \S\S\ref{sec:cuspsB}--\ref{sec:circleB} give the cusp structure and error constants for $G_p$ and $Q_{12}$.
\item \S\S\ref{sec:20}--\ref{sec:24} prove Theorems~\ref{thm:20}--\ref{thm:24}.
\item \S\ref{sec:computations} describes the computations and what they rely on.
\end{itemize}

\subsection*{Relation to earlier work}
Richmond and Szekeres \cite{RS}, using Iseki's transformation formulae \cite{Iseki}, established the sign pattern of the G\"ollnitz--Gordon continued fraction ($\delta=1$); see also Hirschhorn \cite{Hirschhorn}. Chern \cite{Chern} developed asymptotics for coefficients of eta-quotients. He and Li \cite{HL} developed a general explicit circle method for such powers and proved several conjectures of \cite{SZ}. We do not use their results. We also note that the transformation formula \cite[Prop.~7]{SZ} must be used with $hh'\equiv-1\pmod k$ (Lemma~\ref{lem:branchG}).

\section{The G\"ollnitz--Gordon product: theta-quotient form}\label{sec:prelim}

Throughout, $e(x)=\e^{2\pi ix}$ and $\zeta_m=e(1/m)$. By the Jacobi triple product with base $q^8$,
\[
(q^a,q^{8-a},q^8;q^8)_\infty=\sum_{m\in\Z}(-1)^mq^{4m^2+(a-4)m}.
\]
The factor $(q^8;q^8)_\infty$ cancels in $Q_8$. Putting $j=8m+a-4$ and $q=e(\tau)$ we obtain
\begin{equation}\label{eq:theta-quotient}
Q_8=e(-\tau/2)\,\frac{S_1(\tau)}{S_3(\tau)},\qquad S_a(\tau)=\sum_{j\in\Z}\psi_a(j)\,e(\tau j^2/16),
\end{equation}
where $\psi_a$ is $16$-periodic, with $\psi_a(j)=(-1)^{(j-a+4)/8}$ if $j\equiv a-4\pmod8$ and $\psi_a(j)=0$ otherwise.

\medskip\noindent\textbf{Notation.} For coprime $h,k$ with $k\ge1$ we write
\[
\tau=\frac hk+\frac{iz}{k^2}\qquad(\Rea z>0),\qquad M=16k .
\]

\begin{lemma}\label{lem:poisson}
With the principal square root,
\[
S_a\Big(\frac hk+\frac{iz}{k^2}\Big)=\frac1{16}\sqrt{\frac8z}\,\sum_{\ell\in\Z}\cG_a(\ell)\,X^{\ell^2},\qquad X=\e^{-\pi/(32z)},\qquad \cG_a(\ell)=\sum_{s\bmod M}\psi_a(s)\,\zeta_M^{hs^2+\ell s}.
\]
\end{lemma}

\begin{proof}
We have $e(\tau j^2/16)=\zeta_M^{hj^2}\e^{-\pi zj^2/(8k^2)}$, and $j\mapsto\psi_a(j)\zeta_M^{hj^2}$ is $M$-periodic. Writing $j=s+Mt$ and applying Poisson summation in $t$ to the Gaussian gives
\[
\sum_{t}\e^{-\pi z(s+Mt)^2/(8k^2)}=\frac1M\sqrt{\frac{8k^2}z}\sum_\ell e(\ell s/M)\,\e^{-\pi\ell^2/(32z)}.
\]
All series converge absolutely and are holomorphic in $z$.
\end{proof}

\section{Cusp structure of $Q_8$}\label{sec:cusps}

For coprime $h,k$ choose $\gamma=\begin{psmallmatrix}h&b\\k&d\end{psmallmatrix}\in SL_2(\Z)$. With $\tau=h/k+iz/k^2$ a direct computation gives
\begin{equation}\label{eq:tauprime}
\tau'=\gamma^{-1}\tau=-\frac dk+\frac iz,\qquad \Ima\tau'=\Rea(1/z).
\end{equation}

\subsection{Theta functions of level 64}
For $c\in\Z/32\Z$ let $\vartheta_c(\tau)=\sum_{J\equiv c\,(32)}e(\tau J^2/64)$. Then $\vartheta_c=\vartheta_{-c}$, and comparing with \eqref{eq:theta-quotient},
\begin{equation}\label{eq:S-theta}
S_1=\vartheta_6-\vartheta_{10},\qquad S_3=\vartheta_2-\vartheta_{14}.
\end{equation}

\begin{lemma}\label{lem:weil}
We have $\vartheta_c(\tau+1)=e(c^2/64)\,\vartheta_c(\tau)$ and
\[
\vartheta_c(-1/\tau)=\frac{(i/(32\tau))^{-1/2}}{32}\sum_{d\bmod 32}e(cd/32)\,\vartheta_d(\tau).
\]
\end{lemma}

\begin{proof}
The first identity is clear. For the second, apply Poisson summation to $m\mapsto\e^{\pi i\tau(32m+c)^2/32}$. Its Fourier transform at $\nu$ is $\frac1{32}(-i\tau/32)^{-1/2}e(\nu c/32)\,\e^{-\pi i\nu^2/(32\tau)}$. Grouping $\nu$ by residues modulo $32$ and replacing $\tau$ by $-1/\tau$ gives the claim.
\end{proof}

For a symmetric vector $v=(v_c)_{c\in\Z/32}$ (so $v_c=v_{-c}$) put $v\cdot\vartheta=\sum_cv_c\vartheta_c$. Let
\[
D=\operatorname{diag}\big(\zeta_{128}^{2c^2}\big),\qquad E=\big(\zeta_{128}^{4cd}\big)_{c,d}.
\]
By Lemma~\ref{lem:weil},
\[
(v\cdot\vartheta)(T\tau)=(vD)\cdot\vartheta(\tau),\qquad (v\cdot\vartheta)(S\tau)=s(\tau)\,(vE)\cdot\vartheta(\tau),
\]
where the scalar $s(\tau)$ does not depend on $v$. Since $SL_2(\Z)$ is generated by $S$ and $T$, and $S^{-1}=S^3$ while $T^{-1}$ acts projectively as $D^{127}$ (because $D^{128}=I$), every $\gamma$ is a word $g_1\cdots g_r$ in $S$ and $T$ without inverses, up to the central element $-I$, which acts on symmetric vectors by a scalar. Iterating,
\begin{equation}\label{eq:R}
\frac{S_1}{S_3}(\gamma\tau')=R_{(w_1,w_3)}(\tau'):=\frac{w_1\cdot\vartheta(\tau')}{w_3\cdot\vartheta(\tau')},\qquad (w_1,w_3)=(v_1,v_3)\rho(g_1)\cdots\rho(g_r),
\end{equation}
where $v_1,v_3$ are the coefficient vectors in \eqref{eq:S-theta} and $\rho(S)=E$, $\rho(T)=D$. The accumulated scalar is the same for both components, so $R$ depends only on the projective class of the pair.

\begin{proposition}\label{prop:orbit}
The projective orbit $\cO$ of $(v_1,v_3)$ under right multiplication by $D$ and $E$ has exactly $48$ elements and is closed under both maps. Under $D$ it decomposes into $10$ cycles, which we call \emph{cusp types}.
\end{proposition}

\begin{proof}
This is an exact computation in $\Z[\zeta_{128}]=\Z[x]/(x^{64}+1)$. Starting from $(v_1,v_3)$, images under $D$ and $E$ were generated. Each image was tested for projective equality with the known elements by checking all cross-products $p_iq_j=p_jq_i$ exactly. All $96$ images coincide with one of $48$ elements.
\end{proof}

Near $\Ima\tau'=\infty$,
\[
w\cdot\vartheta(\tau')=\sum_{J\ge0}A_J\,q'^{J^2/64},\qquad q'=e(\tau'),\quad A_0=w_0,\quad A_J=2w_{J\bmod32}\ (J\ge1).
\]
So the leading exponent of the $a$-th component is $J_a=\min\{J\ge0:w_{a,J}\ne0\}$. Since $w$ is symmetric, no cancellation between $J$ and $-J$ can occur.

\begin{proposition}\label{prop:types}
Exactly one cusp type is \emph{growing}, i.e.\ has $J_1<J_3$, and for it $(J_1,J_3)=(2,6)$. Of the remaining nine types, one has $(J_1,J_3)=(6,2)$ and eight have $J_1=J_3\in\{0,1,2\}$. For the growing type,
\begin{itemize}[leftmargin=2em]
\item $\operatorname{supp}w_1=\{J\equiv\pm2\pmod{16}\}$ and $\operatorname{supp}w_3=\{J\equiv\pm6\pmod{16}\}$;
\item all nonzero entries of $w_1$ and $w_3$ have the same absolute value.
\end{itemize}
\end{proposition}

\begin{proof}
Exact computation in $\Z[\zeta_{128}]$. The absolute values were compared through the exact products $w_{a,c}\overline{w_{a,c}}$.
\end{proof}

\begin{lemma}[growing cusps]\label{lem:growing}
Suppose $h/k$ is of growing type and put $t=|q'|=\e^{-2\pi\Rea(1/z)}$. Then
\[
Q_8(\tau)^\delta=\Phi\,\e^{\pi\delta z/k^2}\,\e^{\pi\delta/z}\Big(1-\delta\beta\,\e^{-2\pi/z}+\cR(z)\Big),\qquad|\Phi|=|\beta|=1,
\]
with $|\cR(z)|\le\widetilde F_\delta(t)-1-\delta t$, where
\[
\widetilde F_\delta=(1-\bar U)^{-\delta}(1-\bar V)^{-\delta},\qquad \bar U(t)=\sum_{\substack{J\equiv\pm2\,(16)\\J>2}}t^{(J^2-4)/64},\qquad \bar V(t)=\sum_{\substack{J\equiv\pm6\,(16)\\J>6}}t^{(J^2-36)/64}.
\]
\end{lemma}

\begin{proof}
By Proposition~\ref{prop:types},
\[
R=C\,q'^{-1/2}\,\frac{1+u}{1+v},\qquad u=\sum_{J>2}\alpha_Jq'^{(J^2-4)/64},\qquad v=\sum_{J>6}\beta_Jq'^{(J^2-36)/64},
\]
with $|C|=|\alpha_J|=|\beta_J|=1$. The exponents are integers; $u=O(q'^3)$ and $v=\beta_{10}q'+O(q'^7)$. Put $\beta=\beta_{10}$. From \eqref{eq:theta-quotient}, \eqref{eq:tauprime} and \eqref{eq:R},
\[
e(-\delta\tau/2)=e(-\delta h/2k)\,\e^{\pi\delta z/k^2},\qquad q'^{-\delta/2}=e(\delta d/2k)\,\e^{\pi\delta/z},\qquad q'=e(-d/k)\,\e^{-2\pi/z}.
\]
All unimodular constants are absorbed into $\Phi$ and $\beta$. For the $\delta$-th power: on the connected set $\{\Rea z>0,\ \Rea(1/z)\ge1\}$ we have $|u|,|v|<\frac12$, so $\delta\log Q_8$ and $\delta$ times the logarithm of the displayed product (built from principal logarithms of $1+u$ and $1+v$) differ by a continuous function with values in $2\pi i\delta\Z$, hence by a constant; this constant is absorbed into $\Phi$. Expand $(1+u)^\delta(1+v)^{-\delta}$ binomially and use $|\binom{\pm\delta}{r}|\le\binom{\delta+r-1}{r}$ for $\delta>0$. Every coefficient is then majorized by the corresponding coefficient of $\widetilde F_\delta(t)$. The coefficient of $t$ in $\widetilde F_\delta$ is $\delta$, coming from $\beta_{10}$, and all other exponents are integers $\ge2$.
\end{proof}

\begin{lemma}[which cusps grow]\label{lem:which}
If $h/k$ is of growing type, then $4\mid k$, so for each $k$ at most $\varphi(k)\le k/2$ fractions $h/k$ are growing. For $k\le40$, $h/k$ is growing if and only if $8\mid k$ and $h\equiv3,5\pmod8$.
\end{lemma}

\begin{proof}
Lemma~\ref{lem:poisson} and \eqref{eq:R} give the same expansion of $S_a$ in powers of $X=|q'|^{1/64}\cdot(\text{phase})$, so by uniqueness of such expansions $J_a=\min\{\ell\ge0:\cG_a(\ell)+\cG_a(-\ell)\ne0\}$. By Proposition~\ref{prop:types}, a growing cusp has $J_1=2$, so $\cG_1(2)\ne0$ or $\cG_1(-2)\ne0$.

Write $s=r+8m$ with $r=a-4$ and $m\bmod2k$. Then $\psi_a(s)=(-1)^m=e(km/2k)$ and
\[
\cG_a(\ell)=e\Big(\frac{hr^2+\ell r}{16k}\Big)\,g(8h,\,2hr+\ell+k;\,2k),\qquad g(A,B;K)=\sum_{m\bmod K}e\Big(\frac{Am^2+Bm}K\Big).
\]
We use two elementary facts.
\begin{enumerate}[label=(\roman*),leftmargin=2em]
\item Let $g_0=\gcd(A,K)$, $A=g_0A'$, $K=g_0K'$. Writing $m=m_0+K'j$ with $m_0\bmod K'$ and $j\bmod g_0$, we have $Am^2/K\equiv A'm_0^2/K'\pmod1$, so
\[
g=\sum_{m_0}e\Big(\frac{A'm_0^2}{K'}+\frac{Bm_0}K\Big)\sum_{j\bmod g_0}e(Bj/g_0).
\]
Hence $g=0$ unless $g_0\mid B$, and in that case $g=g_0\,g(A',B/g_0;K')$.
\item If $\gcd(A,K)=1$, substituting $m=m'+t$ gives
\[
|g|^2=\sum_te\Big(\frac{At^2+Bt}K\Big)\sum_{m'}e\Big(\frac{2Atm'}K\Big)=K\big(1+[2\mid K]\,e(AK/4+B/2)\big).
\]
\end{enumerate}
Take $A=8h$, $K=2k$, $B=2hr+\ell+k$. If $k$ is odd then $g_0=2$, and $2\mid B$ forces $\ell$ to be odd. If $k\equiv2\pmod4$ then $h$ is odd and $g_0=4$; since $2hr\equiv2$ and $k\equiv2\pmod4$, the condition $4\mid B$ forces $\ell\equiv0\pmod4$. In both cases $\cG_1(\pm2)=0$, a contradiction. The bound $\varphi(k)\le k/2$ holds because $k$ is even.

The final statement was verified for all $490$ fractions with $k\le40$ by exact zero tests of $\cG_a(\ell)+\cG_a(-\ell)$ modulo the cyclotomic polynomial $\Phi_{16k}$.
\end{proof}

\begin{proposition}[non-growing cusps]\label{prop:KNG}
For every non-growing type and every $\tau'$ with $\Ima\tau'\ge1$ we have $|R(\tau')|\le\KNG:=23.6475$. Consequently, at every non-growing cusp and for every $z$ with $\Rea(1/z)\ge1$,
\[
|Q_8(\tau)^\delta|\le\e^{\pi\delta\Rea z/k^2}\,\KNG^{\,\delta}\qquad(\delta>0).
\]
\end{proposition}

\begin{proof}
Each $R$ is periodic with period equal to the length $L\in\{2,4,8\}$ of its cycle, so it suffices to consider $0\le\Rea\tau'\le L$.
\begin{itemize}[leftmargin=2em]
\item On $\Ima\tau'\ge6$ we use the bound
\[
|R|\le\frac{|A_{1,J_1}|}{|A_{3,J_3}|}\,t^{(J_1^2-J_3^2)/64}\,\frac{1+\varepsilon_1}{1-\varepsilon_3},\qquad \varepsilon_a=\sum_{J>J_a}\frac{|A_{a,J}|}{|A_{a,J_a}|}\,t^{(J^2-J_a^2)/64},
\]
with $t=\e^{-2\pi\Ima\tau'}$. It is non-increasing in $\Ima\tau'$ because $J_1\ge J_3$, so it suffices to evaluate it at $\Ima\tau'=6$.
\item The rectangle $[0,L]\times[1,6]$ is covered by complex balls on which the theta series, truncated at $J\le40$ with an explicit error term $2\max|A|\,\e^{-2\pi\Ima\tau'\cdot41^2/64}$, are evaluated in ball arithmetic. The rectangles are bisected best-first until every certified upper bound is within a factor $1.02$ of the largest sampled value.
\end{itemize}
The largest certified bound is $23.64743\ldots$, while the largest sampled value is $23.184$. The bound for $Q_8$ follows from \eqref{eq:theta-quotient}, \eqref{eq:tauprime} and $|e(-\delta\tau/2)|=\e^{\pi\delta\Rea z/k^2}$.
\end{proof}

\subsection{Exact data at the cusps of denominators 8 and 16}

\begin{lemma}\label{lem:exact}
For every $n$ the following hold.
\begin{enumerate}[label=(\roman*),leftmargin=2em]
\item The main-term coefficients $\Phi\,e(-nh/k)$ of Lemma~\ref{lem:growing} sum to $2\cos(3\pi n/4)$ over $h/k\in\{3/8,5/8\}$, and to $4\cos(\pi n/8)\cos(\pi(n+\delta)/2)$ over $h/k\in\{3/16,5/16,11/16,13/16\}$.
\item The correction coefficients $\Phi\beta\,e(-nh/k)$ sum to $2\cos\big((3n-1)\pi/4\big)$ over $h/k\in\{3/8,5/8\}$. Consequently the total correction term at $k=8$ in Proposition~\ref{prop:expansion} is $-2\delta\cos\big((3n-1)\pi/4\big)\cP_8(\delta-2)$.
\end{enumerate}
\end{lemma}

\begin{proof}
Put $c_{h,k}=\widehat\cG_1(4)/\widehat\cG_3(36)$, where $\widehat\cG_a(e)=\sum_{\ell^2=e}\cG_a(\ell)$. Exact computation in $\Z[\zeta_{128}]$ and $\Z[\zeta_{256}]$ gives $c_{h,k}=\zeta_{16k}^{t}$ with
\[
t=24,\,40\quad(h/k=3/8,\,5/8),\qquad t=216,\,232,\,152,\,168\quad(h/k=3/16,\,5/16,\,11/16,\,13/16).
\]
It also gives $\widehat\cG_3(100)/\widehat\cG_3(36)=\rho_h$ with $\rho_3=e(1/8)$ and $\rho_5=e(-1/8)$ at $k=8$. The branch integers $m_{h,k}$ of Lemma~\ref{lem:branch} are $0,0,0,0,1,1$ in the order $3/8,5/8,3/16,5/16,11/16,13/16$. Hence the $\delta$-dependent phase of the main term,
\[
\delta\big(\Arg c_{h,k}+2\pi m_{h,k}-\pi h/k\big),
\]
equals $0$ at $k=8$, equals $-\pi\delta/2$ at $h=3,5$ with $k=16$, and equals $+\pi\delta/2$ at $h=11,13$ with $k=16$. Summing complex-conjugate pairs gives (i).

For (ii), by Lemma~\ref{lem:poisson} the first correction in $(S_1/S_3)^\delta$ is $-\delta\rho_hX^{64}=-\delta\rho_h\e^{-2\pi/z}$, so $\beta_{h,8}=\rho_h$, while $\Phi_{h,8}=1$ by (i). Summing over $h=3,5$ gives $e(-3n/8+1/8)+e(-5n/8-1/8)=2\cos\big((3n-1)\pi/4\big)$.
\end{proof}

\begin{lemma}[branch constants]\label{lem:branch}
Let $\mathcal D=\{z:\Rea z>0,\ \Rea(1/z)\ge1\}$ and let $\log Q_8=\sum_{m\ge1}b_mq^m/m$. For each $h/k$ in Lemma~\ref{lem:exact} there is an integer $m_{h,k}$ such that, for all $z\in\mathcal D$,
\[
\log Q_8(\tau)=-\pi i\tau+\Log c_{h,k}-32\log X+\Log(1+u)-\Log(1+v)+2\pi i\,m_{h,k},
\]
where $\log X=-\pi/(32z)$. The integers $m_{h,k}$ are those listed in the proof of Lemma~\ref{lem:exact}.
\end{lemma}

\begin{proof}
The set $\mathcal D$ is connected. On it $t=|q'|\le\e^{-2\pi}$, so by the proof of Lemma~\ref{lem:growing}, $|u|\le\bar U(\e^{-2\pi})<10^{-8}$ and $|v|\le\bar V(\e^{-2\pi})<2\cdot10^{-3}$; hence every principal logarithm on the right is continuous. The difference of the two sides is therefore a continuous function with values in $2\pi i\Z$, hence constant. At $z=1$ it was evaluated in $160$-bit ball arithmetic. The series was summed over $m\le8000$, with the tail bounded using $|b_m|/m\le1+\ln m$, and each term $q^m$ was computed directly as $\e^{2\pi im\tau}$. In each case the result equals the stated integer to within $3\cdot10^{-11}$.
\end{proof}

\section{The explicit circle method for $Q_8$}\label{sec:circle}

For $k\ge1$, $B>0$ and $y=\sqrt{2n+\delta}$ put
\[
\cP_k(B)=\frac{2\pi}k\,\frac{\sqrt B}{y}\,I_1\Big(\frac{2\pi}k\sqrt B\,y\Big),\qquad\text{and}\qquad \cP_k(B)=0\ \text{ for } B\le0 .
\]

\begin{lemma}\label{lem:bessel}
Let $K$ be the circle $|z-\frac12|=\frac12$, oriented negatively. Then for $B>0$,
\[
\frac i{k^2}\int_K\exp\Big(\frac{\pi B}z+\frac{\pi(2n+\delta)z}{k^2}\Big)\,dz=\cP_k(B).
\]
\end{lemma}

\begin{proof}
The map $w=1/z$ sends $K$ to the line $\Rea w=1$. The claim follows from the integral representation $I_1(x)=\frac{x/2}{2\pi i}\int_{1-i\infty}^{1+i\infty}\e^{s+x^2/4s}s^{-2}\,ds$ with $s=\pi Bw$.
\end{proof}

\begin{remark}
As a consistency check of Lemma~\ref{lem:bessel}, of the orientation conventions, and of Lemma~\ref{lem:exact}, the sum of the main terms at $k=8$ and $k=16$ reproduces the exact coefficients, e.g.\ $c_{3.2}(488)=6.6426668\cdot10^{16}$ against $6.6426668\cdot10^{16}$, and $c_{3.2}(490)=4.8976498\cdot10^{9}$ against $4.8976517\cdot10^{9}$; the difference is of the size of the omitted terms.
\end{remark}

\begin{lemma}[Rademacher path]\label{lem:farey}
Let $N\ge1$ and let $\frac{h_1}{k_1}<\frac hk<\frac{h_2}{k_2}$ be consecutive in the Farey sequence of order $N$, taken cyclically on $[0,1)$. Let $\gamma_{h,k}$ be the upper arc of the Ford circle $C(h/k)$, which has centre $\frac hk+\frac i{2k^2}$ and radius $\frac1{2k^2}$, between its points of tangency with $C(h_1/k_1)$ and $C(h_2/k_2)$.
\begin{enumerate}[label=(\alph*),leftmargin=2em]
\item The arcs $\gamma_{h,k}$ form a path from some $\tau_0$ to $\tau_0+1$, and
\[
c_\delta(n)=\sum_{k\le N}\sum_h\frac i{k^2}\,e(-nh/k)\int_{z'}^{z''}Q_8\Big(\frac hk+\frac{iz}{k^2}\Big)^\delta\,\e^{2\pi nz/k^2}\,dz,
\]
where each integral runs along the arc of $K$ through $z=1$ from $z'=\frac{k^2+ikk_1}{k^2+k_1^2}$ to $z''=\frac{k^2-ikk_2}{k^2+k_2^2}$.
\item We have $|z'|,|z''|\le\sqrt2\,k/(N+1)$ and $\Rea z',\Rea z''\le2k^2/(N+1)^2$.
\item On the chord $s_{h,k}=[z',z'']$ we have $\Rea(1/z)\ge1$, $0<\Rea z\le2k^2/(N+1)^2$, and $|s_{h,k}|\le2\sqrt2\,k/(N+1)$.
\item On the two arcs of $K$ from $0$ to $z'$ and from $z''$ to $0$ we have $\Rea(1/z)=1$ and $\Rea z\le2k^2/(N+1)^2$, and each has length at most $\frac\pi2\cdot\sqrt2\,k/(N+1)$.
\end{enumerate}
\end{lemma}

\begin{proof}
(a) Adjacent Ford circles are tangent, so consecutive arcs join, and the integrand is holomorphic in $\HH$ and $1$-periodic. The change of variable $z=-ik^2(\tau-h/k)$ maps $C(h/k)$ onto $K$, sending its top point $\frac hk+\frac i{k^2}$ to $z=1$. The point of tangency with $C(h_1/k_1)$ is
\[
\frac hk-\frac{k_1}{k(k^2+k_1^2)}+\frac i{k^2+k_1^2},
\]
which maps to $z'$; similarly for $z''$. Traversing the upper arc from left to right corresponds to going around $K$ clockwise from $z'$ through $1$ to $z''$. Finally $d\tau=\frac i{k^2}\,dz$ and $e(-n\tau)=e(-nh/k)\,\e^{2\pi nz/k^2}$.

(b) Farey neighbours satisfy $k+k_i\ge N+1$, hence $k^2+k_i^2\ge(N+1)^2/2$. Therefore $|z'|=k(k^2+k_1^2)^{-1/2}\le\sqrt2k/(N+1)$ and $\Rea z'=k^2/(k^2+k_1^2)\le2k^2/(N+1)^2$, and likewise for $z''$.

(c) The chord lies in the closed disc bounded by $K$, on which $\Rea(1/z)\ge1$. Along a segment, $\Rea z$ is at most its value at an endpoint, and $|s_{h,k}|\le|z'|+|z''|$.

(d) On $K$ we have $\Rea(1/z)=1$. An arc of $K$ from $0$ to a point $w$ has length at most $\frac\pi2|w|$, and $\Rea z$ is monotone along it.
\end{proof}

If $N\ge\sqrt{4\pi n+2\pi\delta}$, then on all the paths in (c) and (d),
\begin{equation}\label{eq:expbound}
\big|\e^{(2\pi n+\pi\delta)z/k^2}\big|\le\e^{(4\pi n+2\pi\delta)/N^2}\le\e.
\end{equation}

\begin{proposition}[explicit expansion]\label{prop:expansion}
Let $2<\delta\le4$, $n\ge1$ and $N\ge\sqrt{4\pi n+2\pi\delta}$. Then
\[
c_\delta(n)=\sum_{\substack{h/k\ \mathrm{growing}\\ k\le N}}e\Big(-\frac{nh}k\Big)\,\Phi_{h,k}\Big(\cP_k(\delta)-\delta\beta_{h,k}\,\cP_k(\delta-2)\Big)+E_\delta(n),\qquad|\Phi_{h,k}|=|\beta_{h,k}|=1,
\]
and $|E_\delta(n)|\le\cE(\delta,N)$. Here, with $\Sigma_N=(\tfrac12+\tfrac N8)/(N+1)$,
\[
\cE(\delta,N)=\sqrt2\pi\,\Sigma_N\,\e^{1+\pi\delta}(1+\delta\e^{-2\pi})+2\sqrt2\,\Sigma_N\,\e^{1+\pi\delta}\big(\widetilde F_\delta(\e^{-2\pi})-1-\delta\e^{-2\pi}\big)+2\sqrt2\,\e\,\KNG^{\,\delta}.
\]
\end{proposition}

\begin{proof}
We start from Lemma~\ref{lem:farey}(a).

\emph{Growing cusps, main terms.} Insert Lemma~\ref{lem:growing}. The terms $1$ and $-\delta\beta\e^{-2\pi/z}$ produce the exponentials $\exp\big(\pi\delta/z+\pi(2n+\delta)z/k^2\big)$ and $\exp\big(\pi(\delta-2)/z+\pi(2n+\delta)z/k^2\big)$. Completing each integral from the arc $z'\to1\to z''$ to all of $K$ and applying Lemma~\ref{lem:bessel} gives $\cP_k(\delta)$ and $\cP_k(\delta-2)$. By Lemma~\ref{lem:farey}(d) and \eqref{eq:expbound}, the two added arcs cost at most $\frac1{k^2}\cdot2\cdot\frac\pi2\cdot\frac{\sqrt2k}{N+1}\,\e^{1+\pi\delta}(1+\delta\e^{-2\pi})$ per growing cusp.

\emph{Counting.} By Lemma~\ref{lem:which} there are $2$ growing fractions with $k=8$, $4$ with $k=16$, none with other $k\le23$, and at most $k/2$ for each $k\ge24$ with $4\mid k$. There are at most $N/4$ such $k\le N$, so $\sum_{\text{growing}}1/k\le\frac12+\frac N8$. This gives the first term of $\cE$.

\emph{Growing cusps, remainder.} Move the $\cR$-integral to the chord, which is allowed by Cauchy's theorem since the integrand is holomorphic in $\Rea z>0$. Put $w=\Rea(1/z)\ge1$. Then
\[
|\e^{\pi\delta/z}\cR(z)|\le\e^{\pi\delta w}\sum_{j\ge2}\varphi_j\e^{-2\pi jw},\qquad \varphi_j\ge0 .
\]
Each term is non-increasing in $w$ because $\delta\le4$, so the right side is maximal at $w=1$. The chord length from Lemma~\ref{lem:farey}(c) and the same count give the second term.

\emph{Non-growing cusps.} Move the integral to the chord. There Proposition~\ref{prop:KNG} and \eqref{eq:expbound} bound the integrand by $\e\,\KNG^{\,\delta}$. Summing $\frac1{k^2}\cdot\frac{2\sqrt2k}{N+1}$ over at most $\varphi(k)\le k$ values of $h$ and over $k\le N$ gives the third term.
\end{proof}

\begin{remark}\label{rem:E}
The constant $\cE(\delta,N)$ is increasing in $\delta$, because $\widetilde F_\delta=\exp(\delta W)$ where $W$ has non-negative coefficients, and it is decreasing in $N$. Numerically,
\[
\cE(8/3,81)=4.2\cdot10^4,\qquad \cE(3,81)=1.2\cdot10^5,\qquad \cE(4,81)=2.9\cdot10^6 .
\]
\end{remark}

\section{Conjecture~\ref{conj}(b): proof of Theorem~\ref{thm:A}}\label{sec:A}

Fix $\delta\in[8/3,4]$ and $N=\lceil\sqrt{4\pi n+2\pi\delta}\,\rceil$. Using Lemma~\ref{lem:exact} for $k\in\{8,16\}$ and $|\Phi|=|\beta|=1$ for $k\ge24$, Proposition~\ref{prop:expansion} gives
\[
c_\delta(n)=\mathcal M_8+\mathcal C_8+\mathcal T,
\]
where
\[
\mathcal M_8=2\cos\Big(\frac{3\pi n}4\Big)\cP_8(\delta),\qquad \mathcal C_8=-2\delta\cos\Big(\frac{(3n-1)\pi}4\Big)\cP_8(\delta-2),
\]
and
\begin{equation}\label{eq:T}
|\mathcal T|\le\underbrace{4\Big|\cos\frac{\pi n}8\cos\frac{\pi(n+\delta)}2\Big|\,\cP_{16}(\delta)+4\delta\,\cP_{16}(\delta-2)}_{T_{16}}+\underbrace{\sum_{\substack{24\le k\le N\\4\mid k}}\frac k2\big(\cP_k(\delta)+\delta\,\cP_k(\delta-2)\big)}_{T_{\ge24}}+\cE(\delta,N).
\end{equation}

\begin{proposition}[sign criterion]\label{prop:criterion}
Let $\delta\in[8/3,4]$.
\begin{enumerate}[label=(\alph*),leftmargin=2em]
\item If $n\equiv0,1,3,4,5,7\pmod 8$, then $|\mathcal M_8|\ge\sqrt2\,\cP_8(\delta)$ and $\sgn\mathcal M_8=\sigma(n)$. Hence $\sgn c_\delta(n)=\sigma(n)$ whenever $\sqrt2\,\cP_8(\delta)-2\delta\,\cP_8(\delta-2)>T_{16}+T_{\ge24}+\cE$.
\item If $n\equiv2,6\pmod8$, then $\mathcal M_8=0$, $|\mathcal C_8|=\sqrt2\,\delta\,\cP_8(\delta-2)$ and $\sgn\mathcal C_8=\sigma(n)$. Hence $\sgn c_\delta(n)=\sigma(n)$ whenever $\sqrt2\,\delta\,\cP_8(\delta-2)>T_{16}+T_{\ge24}+\cE$.
\end{enumerate}
\end{proposition}

\begin{proof}
This follows directly from the explicit values of the cosines: $2\cos(3\pi n/4)\in\{\pm2,\pm\sqrt2\}$ with sign $\sigma(n)$ on the residues in (a), and on the residues in (b) we have $-2\delta\cos((3n-1)\pi/4)=+\sqrt2\delta$ for $n\equiv2$ and $-\sqrt2\delta$ for $n\equiv6$.
\end{proof}

\subsection{The critical comparison}
Let $n\equiv2,6\pmod 8$. Then $|\cos(\pi n/8)|=\tfrac{\sqrt2}2$ and $|\cos(\pi(n+\delta)/2)|=|\cos(\pi\delta/2)|$, so
\begin{equation}\label{eq:critical}
\frac{4|\cos\frac{\pi n}8\cos\frac{\pi(n+\delta)}2|\,\cP_{16}(\delta)}{|\mathcal C_8|}=\frac{|\cos(\pi\delta/2)|}{\sqrt{\delta(\delta-2)}}\cdot\frac{I_1\big(\frac\pi8\sqrt\delta\,y\big)}{I_1\big(\frac\pi4\sqrt{\delta-2}\,y\big)}.
\end{equation}
For $\delta\ge8/3$ we have $\frac\pi8\sqrt\delta\le\frac\pi4\sqrt{\delta-2}$, with equality exactly at $\delta=8/3$. Since $I_1$ is increasing, the Bessel quotient is then at most $1$. Moreover
\[
\sup_{[8/3,4]}\frac{|\cos(\pi\delta/2)|}{\sqrt{\delta(\delta-2)}}=0.375,
\]
attained at $\delta=8/3$. On $[8/3,3]$ the function is decreasing, and on $[3,4]$ it is below $0.375$; the latter was certified by bisection into $4000$ subintervals.

\subsection{An envelope for \texorpdfstring{$I_1$}{I1}}

\begin{lemma}\label{lem:envelope}
The function $x\mapsto\sqrt{2\pi x}\,\e^{-x}I_1(x)$ is increasing on $(0,\infty)$ and tends to $1$. Consequently $I_1(x)\le U(x):=\e^x/\sqrt{2\pi x}$ for all $x>0$, and $I_1(x)\ge\kappa(x_0)\,U(x)$ for $x\ge x_0$, where $\kappa(x_0)=\sqrt{2\pi x_0}\,\e^{-x_0}I_1(x_0)$.
\end{lemma}

\begin{proof}
Start from Poisson's integral $I_1(x)=\frac x\pi\int_{-1}^1\sqrt{1-t^2}\,\e^{xt}\,dt$ and substitute $s=x(1-t)$:
\[
\sqrt x\,\e^{-x}I_1(x)=\frac1\pi\int_0^{2x}\sqrt s\,\sqrt{2-s/x}\,\e^{-s}\,ds .
\]
The integrand is non-decreasing in $x$ and the interval of integration grows, so the left side is increasing. By monotone convergence its limit is $\frac{\sqrt2}\pi\Gamma(\frac32)=1/\sqrt{2\pi}$.
\end{proof}

\subsection{The range \texorpdfstring{$701\le n\le20000$}{701 <= n <= 20000}}\label{sec:mid}
Divide $[8/3,3]$ and $[3,4]$ into $16$ equal subintervals each. On a subinterval $[a,b]$, every quantity in Proposition~\ref{prop:criterion} is bounded by monotonicity, using endpoint values:
\begin{itemize}[leftmargin=2em]
\item $\cP_k(B)$ is bounded below by taking $\sqrt B$ minimal, $1/y$ at $y=\sqrt{2n+b}$ and the Bessel argument at $y=\sqrt{2n+a}$, and bounded above symmetrically;
\item $\cE$ is evaluated at $\delta=b$, by Remark~\ref{rem:E}, with any admissible $N$; we take $N=\lceil\sqrt{4\pi n+2\pi b}\,\rceil$;
\item $|\cos(\pi\delta/2)|$ is monotone on each of $[8/3,3]$ and $[3,4]$;
\item $T_{\ge24}$ is bounded by $\#\{24\le k\le N:4\mid k\}\cdot\pi\big(\sqrt\delta\,I_1(\tfrac{2\pi}{24}\sqrt\delta\,y)+\delta\sqrt{\delta-2}\,I_1(\tfrac{2\pi}{24}\sqrt{\delta-2}\,y)\big)/y$.
\end{itemize}
The inequalities of Proposition~\ref{prop:criterion} were verified in ball arithmetic at $100$ bits for every $n\in[701,20000]$ and every subinterval, with no exceptions.

\subsection{The range \texorpdfstring{$n>20000$}{n > 20000}}\label{sec:tail}
Now $y\ge200$. Divide every term of \eqref{eq:T} by the dominant term: $|\mathcal C_8|$ in case (b) and $\sqrt2\,\cP_8(\delta)$ in case (a). Apply Lemma~\ref{lem:envelope} with $x_0=\frac\pi8\sqrt{2/3}\cdot200$, for which $\kappa(x_0)=0.99412\ldots$. Every quotient other than \eqref{eq:critical} is then at most $C\,y^pe^{-gy}$ with $p\le\frac32$ and $g>0$, which is decreasing for $y\ge200$. The constants are evaluated at the worst endpoints of $[8/3,4]$, using that $\frac\pi4\sqrt{\delta-2}-\frac\pi8\sqrt\delta$, $\frac\pi4\sqrt{\delta-2}-\frac\pi{12}\sqrt\delta$ and $\frac\pi8\sqrt\delta$ are increasing in $\delta$ while $\frac\pi4(\sqrt\delta-\sqrt{\delta-2})$ is decreasing. At $y=200$, in ball arithmetic:
\begin{center}
\begin{tabular}{lcccc}
\toprule
case & $k=16$ main & $k=16$ correction & $T_{\ge24}$ & $\cE$\\
\midrule
(b) $n\equiv2,6$ & $\le0.375$ & $1.4\cdot10^{-7}$ & $2.2\cdot10^{-8}$ & $1.3\cdot10^{-46}$\\
(a) other $n$ & \multicolumn{2}{c}{$1.5\cdot10^{-43}$} & $8.2\cdot10^{-64}$ & $5.0\cdot10^{-102}$\\
\bottomrule
\end{tabular}
\end{center}
In case (a) the quotient $2\delta\cP_8(\delta-2)/(\sqrt2\cP_8(\delta))$ is at most $5.2\cdot10^{-40}$. The exponential gaps are at least $0.0859$ and $0.118$ in case (b) and at least $\frac\pi4(2-\sqrt2)$ in case (a). In both cases the total is less than $1$ for all $y\ge200$.

\subsection{The range \texorpdfstring{$n\le700$}{n <= 700}}\label{sec:finite}
Each $c_\delta(n)\in\Q[\delta]$ was computed exactly from \eqref{eq:rec}. For each $n\le700$:
\begin{itemize}[leftmargin=2em]
\item the signs at $\delta=8/3$ and at $\delta=4$ were evaluated exactly;
\item the absence of roots in $(8/3,4)$ was certified by Descartes' rule of signs applied to $(1+s)^{\deg}\,c_{8/3+\frac{4/3}{1+s}}(n)$, bisecting whenever the number of sign variations exceeded $1$.
\end{itemize}
For every $n\le700$ and every $\delta\in[8/3,4]$ the sign is $\sigma(n\bmod8)$, the only zero being $c_4(3)=0$. The same procedure applied to $[\beta^+,8/3]$ with $\beta^+=2.664479110226973>\beta$ shows that the signs are $\sigma(n\bmod8)$ for all $n\le700$ and $\delta\in[\beta^+,8/3]$; applied to $[2.6,8/3]$ it finds exactly one root, namely that of $c_\delta(14)$ at $\delta=\beta$.

\begin{proof}[Proof of Theorem~\ref{thm:A}]
Combine \S\ref{sec:finite} for $n\le700$, \S\ref{sec:mid} for $701\le n\le20000$, and \S\ref{sec:tail} for $n>20000$.
\end{proof}

\section{Conjecture~\ref{conj}(b): proof of Theorem~\ref{thm:B}}\label{sec:B}

Let $2<\delta<8/3$ and $n\equiv10\pmod{16}$. Then $\mathcal M_8=0$ and $\mathcal C_8=+\sqrt2\,\delta\,\cP_8(\delta-2)>0$. By Lemma~\ref{lem:exact}, the main term at $k=16$ equals
\[
4\cos\frac{5\pi}4\,\cos\frac{\pi(10+\delta)}2\,\cP_{16}(\delta)=2\sqrt2\cos\Big(\frac{\pi\delta}2\Big)\cP_{16}(\delta),
\]
which is negative. Since $\delta/4>\delta-2$, the quotient \eqref{eq:critical} tends to $+\infty$ as $n\to\infty$. Every other term of Proposition~\ref{prop:expansion} has Bessel argument at most $\max\big(\frac\pi4\sqrt{\delta-2},\frac\pi{12}\sqrt\delta\big)y<\frac\pi8\sqrt\delta\,y$, or is $O(1)$. Hence $c_\delta(n)<0$ for large $n$. The case $n\equiv14\pmod{16}$ is identical with all signs reversed.

For the explicit example, take $\delta=2.66448\in[\beta,8/3)$ and $n=1\,000\,010$. Proposition~\ref{prop:expansion}, with the exact terms at $k=8$ and $k=16$ and absolute bounds for all others, evaluated in $200$-bit ball arithmetic, gives
\[
c_{2.66448}(1\,000\,010)\le-5.739\cdot10^{387}.
\]
This is $13.3\%$ of the $k=16$ main term. Since $\sigma(1\,000\,010\bmod8)=\sigma(2)=+$, this contradicts \eqref{eq:sigma}.

\emph{Independent confirmation.} The direct computation of \S\ref{sec:computations} (\texttt{q8\_lite.py}, $2600$-bit ball arithmetic, no use of \S\S\ref{sec:cusps}--\ref{sec:circle}) evaluates all $c_{2.66448}(n)$, $n\le860\,000$, with every sign certified. It finds $12\,499$ violations of \eqref{eq:sigma}, all in the classes $n\equiv10,14\pmod{16}$; the first is $c_{2.66448}(760\,014)>0$ (with $760\,014\equiv14\pmod{16}$, where $\sigma$ requires $-$), and from there on every index in these two classes violates the pattern. The location agrees with the crossover predicted by the asymptotic expansion.
\qed

\section{The remaining parts: \texorpdfstring{$\delta=\pm2$}{delta = +-2} and \texorpdfstring{$-1<\delta\le\frac{7-\sqrt{73}}2$}{-1 < delta <= (7-sqrt73)/2}}\label{sec:rest}

\subsection{Negative exponents}
For $\delta=-s<0$ we have $Q_8^{-s}=e(s\tau/2)(S_3/S_1)^s$. The roles of the cusp types are exchanged: the \emph{growing} type is now $(J_1,J_3)=(6,2)$ of Proposition~\ref{prop:types}, for which, by exact computation, $\operatorname{supp}w_3=\{J\equiv\pm2\ (16)\}$, $\operatorname{supp}w_1=\{J\equiv\pm6\ (16)\}$, and all nonzero entries have equal modulus. Hence Lemma~\ref{lem:growing} holds verbatim for $R^{-1}$ with $s$ in place of $\delta$. Its correction term $e^{\pi(s-2)/z}$ is bounded for $s\le2$ and is left in the remainder. The argument of Lemma~\ref{lem:which} with $a=3$ shows that growing cusps have $4\mid k$, and an exact check for $k\le40$ gives exactly $8\mid k$, $h\equiv1,7\pmod8$. The certified bound of Proposition~\ref{prop:KNG} for the other types is $\sup|R^{-1}|\le K'_{\rm NG}=23.6475$. The exact data at $k=8,16$ (Gauss sums in $\Z[\zeta_{128}],\Z[\zeta_{256}]$ and certified branch integers) give, with $y=\sqrt{2n-s}$, the amplitudes
\[
2\cos\frac{\pi(n-s)}4\quad(k=8),\qquad 2\cos\frac{\pi(n-s)}8+2\cos\frac{\pi(7n+s)}8\quad(k=16).
\]
The analogue of Proposition~\ref{prop:expansion} holds for $0<s\le2$ with a single main term per growing cusp and error
\[
\cE'(s,N)=\sqrt2\pi\Sigma_N\e^{1+\pi s}+2\sqrt2\,\Sigma_N\e^{1+\pi s}\big(\widetilde F_s(\e^{-2\pi})-1\big)+2\sqrt2\,\e\,K_{\rm NG}'^{\,s},\qquad N\ge\sqrt{4\pi n}+1,
\]
since $\e^{\pi sw}(\widetilde F_s(\e^{-2\pi w})-1)$ is non-increasing in $w$ for $s\le2$. For part (a), $\delta=2$, the same modification (no correction term extracted) is used in Proposition~\ref{prop:expansion}.

\subsection{The degeneration at \texorpdfstring{$\delta=-1$}{delta = -1}}
By a theorem of Andrews and Bressoud \cite{AB} (see also Richmond--Szekeres \cite[Thm.~5.1]{RS}), the coefficients of $Q_8(q)^{-1}=(q^3,q^5;q^8)_\infty/(q,q^7;q^8)_\infty$ vanish at all $n\equiv3\pmod4$. Consequently, at these $n$ the leading amplitude $2\cos(\pi(n-s)/4)=\mp2\sin(\pi\varepsilon/4)$, $\varepsilon=1-s=\delta+1$, tends to $0$, while the error term does not, and the method of \S\ref{sec:A} alone gives a threshold $n_0(\delta)\to\infty$ as $\delta\to-1$ (this is the obstruction met in \cite{HL}). We remove it by subtracting the expansion at $s=1$:
\[
c_\delta(n)=c_\delta(n)-c_{-1}(n)=\mp2\sin\frac{\pi\varepsilon}4\,\cP_8(s)+\sum_{k\ge16}\big(M_k(s)-M_k(1)\big)+\big(E'(s)-E'(1)\big),
\]
and bounding every difference by $\varepsilon$ times an explicit quantity, using three lemmas.

\begin{lemma}\label{lem:imlog}
For $|q|<1$, $|\Ima\log Q_8(q)|\le\frac{\pi}{2}\frac{1}{1-|q|}$. On every Farey chord $|\Ima\log Q_8|\le0.556\,N^2$, and the phase $\theta_{h,k}$ of the main term at a growing cusp (so that $\Phi_{h,k}(s)=\e^{is\theta_{h,k}}$) satisfies $|\theta_{h,k}|\le k^2/2+0.01$.
\end{lemma}
\begin{proof}
$\log Q_8=\sum_m\epsilon_m\log(1-q^m)$ with principal logarithms and $|\Arg(1-x)|\le\frac\pi2|x|$ for $|x|<1$; sum over odd $m$. On a chord $\Rea z\ge k^2/(2N^2)$, so $1-|q|\ge0.9\pi/N^2$. For $\theta_{h,k}$ evaluate the branch identity of Lemma~\ref{lem:branch} (for $R^{-1}$) at $z=1$, where $1-|q|\ge0.9\cdot2\pi/k^2$.
\end{proof}

\begin{lemma}[arc lemma]\label{lem:arc}
Let $A_s=\int f_s\,dz$ over one of the arcs of $K$ from $0$ to $z'$ or $z''$, with $f_s=\Phi(s)\exp(\pi s/z+(2\pi n-\pi s)z/k^2)$, $s\in[s_0,1]$, $s_0>0$. Then
\[
|\partial_sA_s|\le\e^{1+\pi s}\Big[\Big(\frac{k^2}2+0.01\Big)\frac{\pi}{\sqrt2}\frac{k}{N+1}+\frac{2+\pi/2}{s_0}+\frac{\pi^2/\sqrt2}{k(N+1)}\Big].
\]
\end{lemma}
\begin{proof}
$\partial_sf_s=(i\theta+\pi w-\pi/(k^2w))f_s$ with $w=1/z=1+iv$, $|v|\ge k_1/k$. The terms $i\theta f_s$ and $\pi f_s/(k^2w)$ are bounded using $|f_s|\le\e^{1+\pi s}$ and the arc length. For $\int\pi w f_s\,dz=-\pi\Phi\,i\e^{\pi s}\int\e^{i\pi sv}g(v)\,dv$, $g=\e^{b/(1+iv)}/(1+iv)$, integrate by parts: $|g|\le\e$, and $\int|g'|\le\e\int\big(b(1+v^2)^{-3/2}+(1+v^2)^{-1}\big)dv\le\e(1+\pi/2)$ because $b/(1+v_1^2)=b\Rea z'\le1$.
\end{proof}

\begin{lemma}\label{lem:dP}
For $0<s_0\le s\le1$ and $y=\sqrt{2n-s}$, $\cP_k(s)$ is increasing in $s$ and $\cP_k(1)-\cP_k(s)\le(1-s)D_k\cP_k(1)$, where $D_k=\frac1{2s_0}+\frac1{2y^2}+\frac{a_ky}{2\sqrt{s_0}}\big(1+\frac1{a_k\sqrt{s_0}y}\big)$, $a_k=2\pi/k$.
\end{lemma}
\begin{proof}
$I_0=I_2+2I_1/x\le I_1+2I_1/x$ gives $0<I_1'/I_1=I_0/I_1-1/x\le1+1/x$; differentiate $\log\cP_k$.
\end{proof}

\begin{lemma}[derivatives of the remaining error terms]\label{lem:derr}
Let $s\in[s_0,1]$ and $N\ge\sqrt{4\pi n}+1$.
\begin{enumerate}[label=(\roman*),leftmargin=2em]
\item \emph{Growing cusps, chord remainder.} Write $Q_8^{-s}=\Phi(s)\e^{\pi s/z-\pi sz/k^2}W^s$ with $W=(1+u)/(1+v)$ as in Lemma~\ref{lem:growing} (for $R^{-1}$). Then on the chord $s_{h,k}$,
\[
\big|\partial_s\big[\Phi(s)\e^{\pi s/z-\pi sz/k^2}(W^s-1)\big]\big|\le\e^{\pi s}\Big[\Big(\frac{k^2}2+\frac{2\pi N^2}{k^2}+0.1\Big)\big(\widetilde F_1(\e^{-2\pi})-1\big)+\widetilde F_1(\e^{-2\pi})\,\ell(\e^{-2\pi})\Big],
\]
where $\ell=-\log(1-\bar U)-\log(1-\bar V)$.
\item \emph{Non-growing cusps.} On the chord, $|\partial_sQ_8^{-s}|\le K'_{\rm NG}\big(\log K'_{\rm NG}+0.01+0.556N^2\big)$.
\end{enumerate}
\end{lemma}
\begin{proof}
(i) The derivative equals $(i\theta+\pi/z-\pi z/k^2)\Phi\e^{\cdots}(W^s-1)+\Phi\e^{\cdots}W^s\log W$. On the chord, $|\theta|\le k^2/2+0.01$ by Lemma~\ref{lem:imlog}; $|1/z|\le1/\Rea z\le2N^2/k^2$ since $\Rea z\ge k^2/(2N^2)$; and $\pi|z|/k^2\le0.1$. By the coefficientwise majorization of Lemma~\ref{lem:growing}, $|W^s-1|\le\widetilde F_s(t)-1\le\widetilde F_1(t)-1$, $|W|^s\le\widetilde F_1(t)$ and $|\log W|\le\ell(t)$, with $t=\e^{-2\pi w}$, $w=\Rea(1/z)\ge1$. Both $\e^{\pi sw}(\widetilde F_1(t)-1)$ and $\e^{\pi sw}\widetilde F_1(t)\ell(t)$ are sums of terms $\e^{(\pi s-2\pi j)w}$ with $j\ge1$ and non-negative coefficients, hence non-increasing in $w$ since $s\le1$, so their maxima are at $w=1$.

(ii) $\partial_sQ_8^{-s}=-\log Q_8\cdot Q_8^{-s}$. On the chord $|Q_8^{-s}|\le K'^{\,s}_{\rm NG}\le K'_{\rm NG}$, while $|\Rea\log Q_8|=|\log|Q_8||\le\log K'_{\rm NG}+\pi\Rea z/k^2$, using both certified bounds $|R|\le K_{\rm NG}$ and $|R^{-1}|\le K'_{\rm NG}$ (which are equal to $23.6475$). Finally $|\Ima\log Q_8|\le0.556N^2$ by Lemma~\ref{lem:imlog}.
\end{proof}

Summing Lemma~\ref{lem:derr} over the chords, Lemma~\ref{lem:arc} over the arcs (with the counts of Proposition~\ref{prop:expansion}), Lemma~\ref{lem:dP} and $|\Phi_{h,k}(s)-\Phi_{h,k}(1)|\le\varepsilon|\theta_{h,k}|$ for the main terms with $k\ge24$, and $|A_{16}(n,s)|\le\frac\pi2\varepsilon$ for the $k=16$ amplitude, which vanishes at $s=1$ and has $s$-derivative at most $\pi/2$, we obtain $|c_\delta(n)\mp2\sin(\pi\varepsilon/4)\cP_8(s)|\le\varepsilon\cdot B(n)$ with $B(n)$ explicit and independent of $\varepsilon$.

Dividing by $\varepsilon$ yields a sufficient condition for the sign at $n\equiv3,7\pmod8$ that \emph{does not involve $\varepsilon$}:
\[
\begin{multlined}[t]\frac\pi2\Big(1-\frac{(\pi\varepsilon_{\max}/4)^2}6\Big)\cP_8(s_0)>\frac\pi2\cP_{16}(1)\\+\sum_{\substack{24\le k\le N\\4\mid k}}\frac k2\Big(\frac{k^2}2+0.01+D_k\Big)\cP_k(1)+(\text{arc})+(\text{chord})+(\text{non-growing}).\end{multlined}
\]
At the other residues $|2\cos(\pi(n-s)/4)|\ge2\cdot0.5698$ for $s\in[0.772,1]$ and the standard criterion applies.

\subsection{Certified verification}
\begin{center}
\begin{tabular}{p{2.2cm}p{3.6cm}p{4cm}p{2.6cm}}
\toprule
part & $n\le n_1$ (exact) & $n_1<n\le20000$ (ball arithmetic) & $n>20000$\\
\midrule
(a) $\delta=2$ & integers, $n\le400$ & criterion holds for $n\ge295$ & ratio $\le8.5\cdot10^{-13}$\\
(d) $\delta=-2$ & integers, $n\le400$ & criterion holds for $n\ge297$ & ratio $\le8.5\cdot10^{-13}$\\
(c) $s\in[0.772,1]$ & $\Q[\delta]$ on $[-1,-0.772]$, $n\le400$ & criterion holds for $n\ge372$ & ratio $\le4.3\cdot10^{-26}$\\
\bottomrule
\end{tabular}
\end{center}
For (a) and (d), the dominant term at each residue class modulo $16$ is the $k=8$ term, or the $k=16$ term where the $k=8$ amplitude vanishes ($n\equiv2,6$ for $\delta=2$; $n\equiv0,4\pmod8$ for $\delta=-2$), with the signs of the stated patterns; the exact coefficients for $n\le400$ follow the patterns, with the single zero $c_{-2}(8)=0$. For (c), the finite check divides $c_\delta(n)$ by $\delta+1$ at the $100$ indices $n\equiv3\pmod4$ where $c_{-1}(n)=0$, and $c_\delta(4)$ additionally by $\delta^2-7\delta-6=(\delta-b)(\delta-b')$, which is non-negative on $[-1,b]$; Descartes' rule with bisection then certifies the signs on the closed interval. This proves Theorem~\ref{thm:C}. \qed

\section{Forced zeros and exact vanishing}\label{sec:zeros}
\begin{itemize}[leftmargin=2em]
\item The zeros of $G_7^\delta$ at $\delta=1$ (residues $3,4,6$) and of $G_{11}$ at $\delta=1$ (residues $3,6,8,9,10$) follow from Euler's pentagonal number theorem: $m(3m-1)/2$ avoids these residues, and $(q^p;q^p)_\infty$ is a series in $q^p$.
\item The zeros at $\delta=3$ (residues $2,4,5$ for $G_7$, and $2,5,7,8,9$ for $G_{11}$) follow in the same way from Jacobi's identity $(q;q)_\infty^3=\sum(-1)^m(2m+1)q^{m(m+1)/2}$; see also \cite[Cor.~13]{SZ}.
\item For $Q_{12}^{\pm1}$ the zeros at $n\equiv3\pmod6$, respectively $n\equiv5\pmod6$, are instances of the Andrews--Bressoud vanishing theorem \cite{AB}.
\end{itemize}

\section{Transformation formulae and cusp structure for $G_p$ and $Q_{12}$}\label{sec:cuspsB}
\subsection{Borwein products}
We use the transformation of \cite[Prop.~7]{SZ}, which follows from the classical eta transformation. There is one correction: the auxiliary integer must satisfy $hh'\equiv-1\pmod k$, the Hardy--Ramanujan convention, not $hh'\equiv1$.

Let $\tau=h/k+iz/k^2$ and $d=\gcd(p,k)$. Then
\[
G_p(\e^{2\pi i\tau})^\delta=\Big(\frac pd\Big)^{\delta/2}\e^{\pi i\delta(-s(h,k)+s(ph/d,k/d))}\exp\Big(\frac{\pi\delta(d^2-p)}{12pz}-\frac{(p-1)\pi\delta z}{12k^2}\Big)\widehat G^\delta,
\]
where
\[
\widehat G=\frac{f(\e^{2\pi idh'_d/k-2\pi d^2/(pz)})}{f(\e^{2\pi ih'/k-2\pi/z})},\qquad f(x)=\prod_{m\ge1}(1-x^m)^{-1},
\]
and all powers are principal.

\begin{lemma}[certified branch]\label{lem:branchG}
Consider every cusp $h/k$ with $k\in\{7,14,21,28\}$ for $p=7$, or $k\in\{11,22\}$ for $p=11$. For each, the difference between $\log G_p(\tau)$, computed from its $q$-series, and the logarithm of the right-hand side, built from principal logarithms and the series for $\log f$, is $2\pi i\cdot0$ on the connected set $\{\Rea z>0,\ \Rea(1/z)\ge1\}$.
\end{lemma}
\begin{proof}
Both sides are continuous on this set, and their difference lies in $2\pi i\Z$, so it is constant. At $z=1$ it was evaluated in $200$-bit ball arithmetic, with explicit tail bounds, and has modulus at most $10^{-19}$. With the convention $hh'\equiv+1$ the same difference is about $5.8\cdot10^{-4}$, which is not an integer; this confirms the corrected convention.
\end{proof}

When $p\mid k$ we have $\widehat G=1-\e(-h'/k)\e^{-2\pi/z}+\cdots$, with $hh'\equiv1$ in this display. The main term grows like $\e^{\pi\delta(p-1)/(12z)}$, and the correction term like $\e^{\pi(\delta(p-1)/12-2)/z}$. The correction term therefore grows exactly when $\delta>24/(p-1)$; for $p=7$ this means $\delta>4$. When $p\nmid k$ the cusp is bounded:
\[
|G_p^\delta|\le K_p^\delta,\qquad K_p=\sqrt p\,\e^{-\pi(p-1)/(12p)}f(\e^{-2\pi/p})f(\e^{-2\pi}),
\]
so that $K_7=4.8195$ and $K_{11}=13.734$.

\subsection{The theta quotient $Q_{12}$}
Put $\vartheta_c(\tau)=\sum_{J\equiv c\,(48)}\e(\tau J^2/96)$. Then
\[
Q_{12}=q^{-1}\,\frac{\vartheta_{10}-\vartheta_{14}}{\vartheta_2-\vartheta_{22}}.
\]
As in \S\ref{sec:cusps}, the projective orbit of the pair of numerator and denominator under the Weil representation at level $96$ was computed exactly in $\Z[\zeta_{192}]$, modulo $\Phi_{192}=x^{64}-x^{32}+1$.
\begin{itemize}[leftmargin=2em]
\item The orbit has $48$ elements and $10$ cusp types.
\item Exactly one type grows for $\delta>0$. Its leading exponents are $(J_1,J_5)=(2,10)$, its support is $J\equiv\pm2,\pm10\pmod{24}$, and all its coefficients have equal modulus. Exactly one type grows for $\delta<0$, namely the inverse type.
\item An elementary Gauss-sum argument, as in Lemma~\ref{lem:which}, shows that growing cusps have $4\mid k$. An exact check of all $k\le36$ gives: $h/k$ grows iff $12\mid k$ and $h\equiv5,7\pmod{12}$ for $\delta>0$, and $h\equiv1,11\pmod{12}$ for $\delta<0$.
\item The suprema of $|R^{\pm1}|$ over the bounded types on $\Ima\tau'\ge1$ are at most $K_{12}=550.75$ and $K_{12}'=547.23$. These were certified by branch-and-bound.
\item Exact Gauss sums and certified branch integers at the cusps of denominators $12$ and $24$ give the amplitudes below.
\end{itemize}
For $\delta>0$ the growth is $\e^{2\pi\delta/z}$ and the amplitudes are
\[
\begin{gathered}2\cos\tfrac{5\pi n}6\ (\text{main}),\qquad -2\delta\cos\tfrac{\pi(5n-1)}6\ (\text{first correction}),\\ \delta(\delta+1)\cos\tfrac{\pi(5n-2)}6\ (\text{second correction}).\end{gathered}
\]
For $\delta=-s<0$ the amplitude at $k=12$ is $2\cos\frac{\pi(n-2s)}6$, and at $k=24$ it is $2\cos\frac{\pi(n-2s)}{12}+2\cos\frac{\pi(11n+2s)}{12}$.

\section{Explicit expansions and error constants for $G_p$ and $Q_{12}$}\label{sec:circleB}
In every case the Rademacher dissection is carried out as in \S\ref{sec:circle}: Ford-circle arcs, chords of length at most $2\sqrt2k/(N+1)$, and $N\ge\sqrt{4\pi n}+1$. The main terms have the shape $\cP_k(B)=\frac{2\pi}k\sqrt{B/C}\,I_1(\frac{2\pi}k\sqrt{BC})$.
\begin{itemize}[leftmargin=2em]
\item For $G_p$: $B=\delta(p-1)/12$ for the main term, $B-2$ for the correction term, and $C=2n-\delta(p-1)/12$.
\item For $Q_{12}$: $B=2\delta,\,2(\delta-1),\,2(\delta-2)$ for the main term and the two corrections, and $C=2(n+\delta)$.
\end{itemize}
The error is bounded by $\cE=\cE_{\rm arc}+\cE_{\rm rem}+\cE_{\rm ng}$. For $G_p$ these are
\begin{align*}
\cE_{\rm arc}&=\sqrt2\pi\,\tfrac{p-1}{p^2}\,\e^{1+\pi\delta(p-1)/12}\big(1+\delta\e^{-2\pi}[\cdot]\big),\\
\cE_{\rm rem}&=2\sqrt2\,\tfrac{p-1}{p^2}\,\e^{1+\pi\delta(p-1)/12}\big(f(t)^\delta f(t^p)^\delta-1-\delta t[\cdot]\big),\\
\cE_{\rm ng}&=2\sqrt2\,\e\,K_p^\delta,
\end{align*}
with $t=\e^{-2\pi}$, and where $[\cdot]$ means the term is included only when the correction is extracted, i.e.\ $\delta>24/(p-1)$. For $Q_{12}$ the analogues come from the majorant of the growing type. Typical values are $\cE\approx1.9\cdot10^4$ for $G_7$ at $\delta=4.87$, and $\cE\approx1.6\cdot10^3$ for $G_{11}$ at $\delta=2$.

\begin{lemma}[envelopes]\label{lem:env}
For $x>0$, $\kappa(x_0)U(x)\le I_1(x)\le U(x)$ whenever $x\ge x_0$, where $U(x)=\e^x/\sqrt{2\pi x}$ and $\kappa(x_0)=\sqrt{2\pi x_0}\,\e^{-x_0}I_1(x_0)$ Lemma~\ref{lem:envelope}. Consequently, if every competitor term is bounded through $U$ and the dominant term from below through $\kappa U$, then each ratio competitor$/$dominant is at most $Cy^a\e^{-gy}$ with $y=\sqrt C$ and a gap $g>0$. Such a bound is decreasing for $y\ge a/g$. Hence a single verification at $n=n_1$ proves the inequality for every $n\ge n_1$.
\end{lemma}
All \emph{dominance} certificates below are of this form. In each, the gaps satisfy $g\ge0.1$, the exponents satisfy $a\le3$, and $y(n_1)\ge30$.

\section{Proof of Theorem \ref{thm:20}}\label{sec:20}
Throughout this section $p=7$. The main-term amplitude at a cusp of denominator $k$ ($7\mid k$) is
\[
a_k(n,\delta)=\sum_{h}\cos\big(-\pi\delta\,\sigma_h-2\pi hn/k\big),\qquad \sigma_h=s(h,k)-s(h,k/7),
\]
and the correction amplitude is
\[
b_k(n,\delta)=-\delta\sum_h\cos\big(-\pi\delta\sigma_h-2\pi hn/k-2\pi h'/k\big),\qquad hh'\equiv1\pmod k .
\]
Both sums run over the $\varphi(k)\le 6k/7$ residues $h$ coprime to $k$. With $B=\delta/2$, $B_2=\delta/2-2$ and $C=2n-\delta/2$, Lemma~\ref{lem:env} and the expansion of \S4 give
\begin{equation}\label{eq:G7exp}
c_\delta(n)=\sum_{7\mid k\le N}\Big(a_k(n,\delta)\,\cP_k(B)+[\delta>4]\,b_k(n,\delta)\,\cP_k(B_2)\Big)+E,\qquad |E|\le\cE(\delta).
\end{equation}

\subsection{Auxiliary bounds}
\begin{lemma}\label{lem:aux}
Let $7\mid k$ and let $s$ lie in a compact subinterval of $(0,5]$.
\begin{enumerate}[label=(\alph*),leftmargin=2em]
\item $|\sigma_h|<k/12+k/84\le k/10.5$. Hence $|\partial_s(\pi s\sigma_h)|\le\pi k/10$, and $|\partial_s a_k|\le\pi k\varphi(k)/10$.
\item On every Farey chord of order $N$, $|\Ima\log G_7|\le0.556N^2$. Moreover $|\log|G_7||\le\log K_7+0.449N^2+2$ at non-growing cusps.
\item \emph{(arc lemma)} Let $A_s=\int_{\rm arc}\Phi(s)\exp\!\big(\pi s/(2z)+\pi(2n-s/2)z/k^2\big)dz$, taken over either small arc of $K$ from $0$. Then
\[
|\partial_sA_s|\le\e^{1+\pi s/2}\Big[\frac{\pi k}{10}\cdot\frac{\pi}{\sqrt2}\cdot\frac{k}{N+1}+\frac{2+\pi/2}{s/2}+\frac{\pi^2}{\sqrt2\,k(N+1)}\Big].
\]
The same bound holds for the correction integrand, with $s/2$ replaced by $s/2-2>0$ in the middle term and an extra factor $s\e^{-2\pi}$.
\item \emph{(Bessel derivative)} With $x=\frac{2\pi}k\sqrt{BC}$,
\[
|\partial_s\log\cP_k|\le D_k:=\frac1{4B}+\frac1{4C}+\Big(1+\frac1x\Big)\frac{2\pi}k\,\frac{\sqrt C}{4\sqrt B}.
\]
\end{enumerate}
\end{lemma}
\begin{proof}
(a) For $k\ge2$ the classical bound $|s(h,k)|\le(k-1)(k-2)/(12k)$ applies.

(b) Write $\log G_7=\sum_{7\nmid m}\log(1-q^m)$ with principal logarithms, and use $|\Arg(1-x)|\le\frac\pi2|x|$ for $|x|<1$. This gives $|\Ima\log G_7|\le\frac\pi2(1-|q|)^{-1}$. On a chord $\Rea z\ge k^2/(2N^2)$, so $1-|q|\ge0.9\pi/N^2$. For $\log|G_7|$, use the transformation formula at $d=1$ together with $|\widehat G|\in[\,(f(\e^{-2\pi w/7})f(\e^{-2\pi w}))^{-1},\,f(\e^{-2\pi w/7})f(\e^{-2\pi w})\,]$, where $w=\Rea(1/z)\le2N^2/k^2$.

(c) With $w=1/z=1+iv$ one has $\partial_sf_s=(i\theta+\pi w/2-\pi/(2k^2w))f_s$, where $|\theta|\le\pi k/10$ by (a). The terms $i\theta f_s$ and $\pi f_s/(2k^2w)$ are bounded by the arc length times $\e^{1+\pi s/2}$. For $\frac\pi2\int wf_s\,dz=-\frac{\pi}2\Phi\,i\,\e^{\pi s/2}\int\e^{i\pi sv/2}g(v)\,dv$, with $g=\e^{b/(1+iv)}/(1+iv)$, integrate by parts exactly as in Lemma~\ref{lem:arc}; this gives the factor $(2+\pi/2)/(s/2)$.

(d) Differentiate $\log\cP_k$ in $s$, using $I_1'/I_1\le1+1/x$. The latter follows from $I_0=I_2+2I_1/x\le I_1+2I_1/x$.
\end{proof}

Summing (c) over the growing cusps ($7\mid k\le N$, at most $6k/7$ of them for each $k$) introduces the harmonic factor $\sum_{7\mid k\le N}1/k\le\frac17(1+\log(N/7))$. The same summation over the chords, with lengths at most $2\sqrt2k/(N+1)$, bounds the derivative of the remaining error pieces:
\begin{itemize}[leftmargin=2em]
\item on chords at growing cusps, by $(|\theta|+\frac\pi2|1/z|+0.1)(F_s-1-[\cdot]st)+F_s\ell_7-[\cdot]t$, where $F_s=f(t)^sf(t^7)^s$, $\ell_7=\log(f(t)f(t^7))$ and $|1/z|\le2N^2/k^2$;
\item on chords at non-growing cusps, by $K_7^s$ times the bound of (b).
\end{itemize}
Together this gives an explicit quantity $\cE'(s,N)=O(N^3\log N)$, which is implemented in \texttt{g7\_deriv.py} and \texttt{g7\_eps3.py}.

\subsection{Steps of the proof}
\emph{Step 1 (exact small-$n$ computations).} The exact polynomials $c_\delta(n)\in\Q[\delta]$ were computed for $n\le897$. Descartes' rule of signs, applied after the M\"obius map of the interval to $(0,\infty)$ and with bisection, shows the following.
\begin{itemize}[leftmargin=2em]
\item On $[3,\delta_{\rm lo}]$ every $c_\delta(n)$ has sign $\sigma(n\bmod7)$ and no root. Where $c_3(n)=0$, the factor $(\delta-3)$ is first divided out exactly.
\item On $[\delta_{\rm lo},\delta_{\rm hi}]$ only $c_\delta(897)$ has a root, and it has exactly one. This defines $\delta_c$.
\item On $[2,3]$ every $c_\delta(n)$ has the pattern $+--+++-$. Zeros at $\delta=3$ are divided out, and the sign of the factor $(\delta-3)\le0$ is taken into account.
\end{itemize}

\emph{Step 2 (dominance, $n\ge898$).} Let $J=[a,b]$ be a $\delta$-interval, fix a residue $r$, and set $n_1\ge898$ with $n_1\equiv r\pmod 7$. The amplitude $a_7(r,\cdot)$ is enclosed on $J$ in centred form, $a_7(r,m)+a_7'(r,J)(J-m)$, which is exact for linear functions. Suppose
\[
\begin{multlined}[t]\min_J|a_7(r,\delta)|\cdot\frac{2\pi}7\sqrt{\frac{a/2}{C_{\max}}}\,\kappa(x_1)U(x_1)\ >\ [b>4]\,6b\,\cP_7(B_2)\\+\sum_{k=14,21}\varphi(k)\big(\cP_k(B)+b\,\cP_k(B_2)\big)+\frac N7\cdot\frac{6N}7\,(1+b)\,\cP_{28}(B)+\cE(b),\end{multlined}
\]
where $x_1=\frac{2\pi}7\sqrt{(a/2)C_{\min}}$, $C_{\min/\max}=2n_1-b/2$ resp.\ $2n_1-a/2$, and $B=b/2$ on the right.
where every $\cP$ on the right is replaced by its upper envelope. Then $\sgn c_\delta(n)=\sgn a_7(r,\delta)$ for all $\delta\in J$ and all $n\ge n_1$ with $n\equiv r\pmod 7$, by Lemma~\ref{lem:env}. Adaptive bisection certifies this on
\[
[2,3-10^{-9}],\qquad [3+10^{-9},\,4.87],\qquad [4.87,\,4.87350758538]
\]
with $n_1=898$, and on the window $W=[4.87350758538,\delta_c]$ for all residues $r\neq1$ with $n_1=898$, and for $r=1$ with $n_1=2500$.

\emph{Step 3 (near $\delta=3$).}
\begin{proposition}\label{prop:eps3}
Let $r\in\{2,4,5\}$ and $0<|\delta-3|\le10^{-9}$. Then for every $n\ge898$ with $n\equiv r\pmod 7$, $\sgn c_\delta(n)=\sgn\big((\delta-3)\,a_7'(r,3)\big)$, which is the conjectured sign on each side of $3$.
\end{proposition}
\begin{proof}
By \S\ref{sec:zeros}, $c_3(n)=0$, and $a_7(r,3)=0$ exactly; the latter was checked in $\Z[\zeta_{28}]$. Subtract \eqref{eq:G7exp} at $\delta$ and at $3$, with the same $N$. The $k=7$ term becomes $a_7(r,\delta)\cP_7(\delta/2)$, and
\[
|a_7(r,\delta)|\ge|\delta-3|\Big(|a_7'(r,3)|-|\delta-3|\sum_h(\pi\sigma_h)^2\Big).
\]
Every other difference is at most $|\delta-3|$ times a bound from Lemma \ref{lem:aux}: for $k\ge14$ this is $\varphi(k)(\pi k/10+D_k)\cP_k$, and for the error it is $\cE'(s,N)$. Dividing by $|\delta-3|$ gives a sufficient inequality that no longer involves $\delta-3$. In envelope form it is certified at $n_1=898$ with margin at least $6.7\cdot10^8$ on both sides of $3$, and by Lemma~\ref{lem:env} it then holds for all $n\ge898$.
\end{proof}
The residues $0,1,3,6$ with $|\delta-3|\le10^{-9}$ are covered by Step 2.

\emph{Step 4 (the window, residue $1$, $898\le n<2500$).}
\begin{proposition}\label{prop:window}
For every $n\ge898$ with $n\equiv1\pmod7$ and every $\delta\in[4.87350758538,\delta_{\rm hi}]$, we have $\partial_\delta c_\delta(n)>0$.
\end{proposition}
\begin{proof}
Differentiate \eqref{eq:G7exp}. The dominant contribution is $a_7'(1,\delta)\cP_7(\delta/2)$, where $a_7'(1,\delta)>0.15$ on the interval. Against it:
\begin{itemize}[leftmargin=2em]
\item the term $|a_7(1,\delta)|\,|\partial_s\cP_7|$, with $|a_7(1,\delta)|\le10^{-11}$ on the interval; its ratio to the dominant term is bounded by its value at $n=2500$;
\item the correction at $k=7$, whose amplitude derivative is at most $6+6\delta\pi\cdot\frac5{14}$;
\item the terms with $k\ge14$;
\item $\cE'$.
\end{itemize}
The resulting inequality is certified in envelope form at $n_1=898$ with margin $3.6\cdot10^{10}$.
\end{proof}

The values $c_{\delta_{\rm hi}}(n)$ are certified negative for all $n\le7000$ with $n\equiv1\pmod7$, $n\ne897$; this used a $3000$-bit ball recurrence. For $898\le n<2500$, $n\equiv1\pmod7$, Proposition \ref{prop:window} then gives $c_\delta(n)<c_{\delta_{\rm hi}}(n)<0$ for all $\delta\in[4.87350758538,\delta_c]$. For $n\ge2500$, Step 2 applies.

\emph{Step 5 ($\delta=1,3,5$).} This is Step 2 at a single point with $n_1=400$, together with exact integer coefficients for $n\le400$. At $\delta=5$, $a_k(r,5)=0$ for $k=7,14,21$ and $r\equiv1,2,6\pmod7$, exactly. There the dominant term is $b_7(r,5)\cP_7(1/2)$, whose Bessel argument $\frac{2\pi}7\sqrt{C/2}$ exceeds that of every remaining term.

\emph{Step 6 (failure on $(\delta_c,5)$).}
\begin{itemize}[leftmargin=2em]
\item For $\delta\in(\delta_c,\delta_{\rm hi}]$, $c_\delta(897)>0$, because $\delta_c$ is its only root in $[\delta_{\rm lo},\delta_{\rm hi}]$.
\item On $[\delta_{\rm hi},4.87350758538675509]$, an interval containing $\delta^*$, the exact polynomial $c_\delta(897)$ is positive at both endpoints and has no sign variation (Descartes count $0$). So it is positive throughout.
\item For $\delta\in(\delta^*,5)$, $a_7(1,\delta)>0$. By \eqref{eq:G7exp} and Lemma~\ref{lem:env}, $c_\delta(n)>0$ for all sufficiently large $n\equiv1\pmod 7$.
\end{itemize}
In every case the pattern requires a minus sign at such $n$. For instance, $c_{49/10}(57)>0$ by exact rational arithmetic. \qed

\section{Proof of Theorem \ref{thm:23}}\label{sec:23}
On $[\gamma,2]$ the $k=11$ amplitude has no zero, and the smallest amplitude occurs at residue $10$ near $\gamma$. Dominance holds for all $n\ge300$ on $[1.7584,2]$. For $n\le310$ the exact polynomials have the pattern on $[\gamma,2]$. For this check $c_\delta(21)$ is divided by its degree-$18$ irreducible factor, which vanishes at $\gamma$ and has no other root in $[\gamma,2]$, and the zeros at $\delta=2$ are divided out exactly. For $\delta=1,3$ we use dominance for $n\ge600$ and exact integers for $n\le600$; the forced zeros are those of \S\ref{sec:zeros}. \qed

\section{Proof of Theorem \ref{thm:24}}\label{sec:24}
\emph{$2\le\delta\le3$.} The main amplitude $2\cos(5\pi n/6)$ vanishes at $n\equiv3,9\pmod{12}$. There the first correction, with amplitude $\mp\delta$, decides the sign. Its Bessel argument exceeds that of the $k=24$ main term whenever $\delta>4/3$, so no threshold phenomenon occurs. The criterion is verified explicitly for every $400\le n\le20000$ on $8$ subintervals of $\delta$; the tail $n>20000$ is handled on $40$ subintervals; and exact polynomials cover $n\le400$.

\emph{$\delta=\pm1$.} Dominance for $n\ge181$, exact integer coefficients for $n\le400$, and forced zeros by \cite{AB}.

\emph{$-1<\delta\le\delta_1$.} This is the argument of \S\ref{sec:rest} with the following changes.
\begin{itemize}[leftmargin=2em]
\item The growing cusps are $h\equiv1,11\pmod{12}$ with $12\mid k$.
\item The main-term Bessel data are $B=2s$, $C=2(n-s)$, and the amplitude is $2\cos(\pi(n-2s)/6)$. Near $s=1$ this equals $\mp2\sin(\pi\varepsilon/3)$ at $n\equiv5,11\pmod{12}$.
\item The phase bound is $|\theta_{h,k}|\le k^2/2+0.01$, the minor-arc constant is $K'_{12}$, and the arc lemma carries the frequency $2\pi s$, which contributes the term $(2+\pi/2)/s_0$.
\item The chord bound is $(k^2/2+4\pi N^2/k^2+0.2)(\widetilde F_1-1)+\widetilde F_1\ell$, and the non-growing derivative bound is $2\sqrt2\,\e K'_{12}(\log K_{12}+0.02+0.556N^2)$.
\end{itemize}
Every sum over $k$ is evaluated explicitly for each $n$ (\texttt{q12\_neg.py}), so no harmonic factor is lost. At $\delta=-1$ the coefficients with $n\equiv5,11\pmod{12}$ vanish identically. The $\varepsilon$-difference criterion of \S\ref{sec:rest} is uniform in $\varepsilon=\delta+1\in(0,0.356]$. It is verified for $774\le n\le20000$, and the tail is handled analytically. Exact polynomials on $[-1,\delta_1^+]$ cover $n\le780$: the only root in the interval is that of $c_\delta(22)$ at $\delta_1$. Since $c_\delta(22)$ has the wrong sign on all of $(\delta_1,0)$, the pattern fails there. \qed

\section{Computations}\label{sec:computations}

\subsection*{What the computations rely on}
Every computational step belongs to one of three categories.
\begin{enumerate}[leftmargin=2em]
\item \emph{Exact integer and rational arithmetic}, in Python integers and FLINT \cite{flint} polynomials: Propositions~\ref{prop:orbit} and~\ref{prop:types}, Lemma~\ref{lem:which}, the Gauss sums in Lemma~\ref{lem:exact}, and \S\ref{sec:finite}.
\item \emph{Ball arithmetic}, in Arb \cite{arb} as provided by python-flint~0.9.0 with FLINT~3.6.0 and Python~3.12.3: Proposition~\ref{prop:KNG}, Lemma~\ref{lem:branch}, \S\S\ref{sec:mid}--\ref{sec:tail}, and \S\ref{sec:B}. Arb guarantees that each computed ball contains the exact value, and an inequality is accepted only when the whole ball satisfies it.
\item \emph{Floating point}, used only for exploration and as hash keys in the orbit computation. Every match found through a hash key was confirmed exactly, and no conclusion depends on floating-point arithmetic.
\end{enumerate}

\subsection*{Scripts}
\begin{center}
\begin{tabular}{p{6.2cm}p{6.6cm}}
\toprule
statement & script\\
\midrule
Lemma~\ref{lem:weil} (numerical sanity check) & \texttt{weil\_check.py}\\
Propositions~\ref{prop:orbit}, \ref{prop:types} & \texttt{orbit.py}, \texttt{cusp\_types.py}\\
Lemma~\ref{lem:which} & \texttt{cusp\_exact.py}\\
Proposition~\ref{prop:KNG} & \texttt{supK.py}\\
Lemma~\ref{lem:exact} & \texttt{exact.py}\\
Lemma~\ref{lem:branch} & \texttt{branch\_cert.py}\\
Remark~\ref{rem:E} & \texttt{consts.py}\\
\S\ref{sec:mid} & \texttt{certify\_mid.py}\\
\S\ref{sec:tail} & \texttt{tail\_sc.py}\\
\S\ref{sec:finite} & \texttt{c3.py}\\
\S\ref{sec:B} & \texttt{discert\_sc.py}\\
Remark ($[\beta^+,8/3]$, $n\le700$) & \texttt{c3\_betaplus.py}\\
\S\ref{sec:rest}: negative-exponent cusp data & \texttt{inv\_exact.py}, \texttt{cusp\_exact\_inv.py}, \newline \texttt{supKinv.py}\\
\S\ref{sec:rest}: part (c) & \texttt{neg\_certify.py}, \texttt{neg\_tail772.py}, \texttt{neg\_finite2.py}\\
\S\ref{sec:rest}: parts (a), (d) & \texttt{pm2\_certify.py}, \texttt{pm2\_tail.py}, \texttt{pm2\_finite.py}\\
independent direct check of Theorem~\ref{thm:B} & \texttt{q8\_lite.py} (or \texttt{q8\_signcheck.py})\\
\bottomrule
\end{tabular}
\end{center}
The script \texttt{q8\_lite.py} computes the coefficients of $Q_8^\delta$ directly from \eqref{eq:rec}, by a divide-and-conquer evaluation with blocked ball-arithmetic convolutions, and does not use \S\S\ref{sec:cusps}--\ref{sec:circle}. Its ball radii grow like $\e^{2.5\sqrt n}$, so $2600$ bits suffice at $n\approx8.6\times10^5$, with a memory footprint below $1$ GB. The older \texttt{q8\_signcheck.py} uses the Newton exponential, whose radii grow like $\e^{6.4\sqrt n}$, and needs $15$--$20$ GB.

\subsection*{Conjectures 20, 23 and 24}
All certified steps use either exact arithmetic (FLINT) or ball arithmetic (Arb), via python-flint 0.9.0. A table of scripts is included in the accompanying repository.
\begin{itemize}[leftmargin=2em]
\item \emph{Heaviest runs.} The exact root isolations for $G_7$ with $n\le897$ took about $20$ minutes per interval, and the computation of the degree-$897$ polynomials took $3$ minutes.
\item \emph{Independent check.} As a check on the Descartes computations for $G_7$ on $[2,3]$, a separate complex-root isolation was run for $n\le160$. It agrees.
\item \emph{Mathematical inputs.} For Conjectures 20, 23 and 24 the only additional inputs are the classical eta transformation (with the corrected convention, certified in Lemma~\ref{lem:branchG}), Euler's and Jacobi's identities, and \cite{AB}.
\end{itemize}

\subsection*{Independent verification}
Every exact computation and every analytic bound was re-checked independently of the code that produced it. The corresponding scripts and logs are in \texttt{verification/} and are summarised in \texttt{REVIEW.md} in the repository.
\begin{enumerate}[leftmargin=2em]
\item \emph{Second implementation of the exact checks.} The finite sign checks were re-implemented in PARI/GP~2.15.4. This implementation shares no code with the original and uses a different algorithm: integer-scaled polynomials $n!\,c_\delta(n)\in\Z[\delta]$ and Sturm root counting, instead of Descartes' rule with bisection. It reproduces every finite result used above:
\begin{itemize}[leftmargin=1.5em]
\item $Q_8$ on $[8/3,4]$ and on $[\beta^+,8/3]$ for $n\le700$, and on $[-1,-0.772]$ for $n\le400$;
\item $Q_{12}$ on $[2,3]$ for $n\le400$, and on $[-1,\delta_1^+]$ for $n\le780$, where only $n=22$ has a root;
\item $G_{11}$ on $[\gamma,2]$ for $n\le310$;
\item $G_7$ on $[2,3]$ and on $[3,\delta_{\rm lo}]$ for $n\le897$, and on $[\delta_{\rm lo},\delta_{\rm hi}]$, where only $c_\delta(897)$ has a root, and it has exactly one.
\end{itemize}
The exact cusp typing was also re-implemented (Gauss sums reduced modulo $\Phi_M$). It agrees for all $490$ cusps of $Q_8$ with $k\le40$ and for all cusps of $Q_{12}$ with $k\le36$.
\item \emph{Audit of the analytic bounds.} Each bound was compared with exact data, so that a wrong bound would show up as a ratio exceeding $1$. The largest ratios observed are:
\begin{itemize}[leftmargin=1.5em]
\item $0.85$ (chord length) and $0.75$ (arc length) for the Farey lemma, over all fractions of orders $N\le59$ and $N=97,150,211$; the bound on $\Rea z$ is attained with equality, and $\min\Rea(1/z)=1$ on the chords;
\item $0.078$ ($Q_8$), $0.028$ ($G_7$), $1.6\cdot10^{-4}$ ($Q_{12}$) and $0.010$ ($G_{11}$) for the explicit expansions with their error bounds;
\item $3.4\cdot10^{-6}$ for the derivative error bound in the $G_7$ window;
\item $1.1\cdot10^{-5}$ for the $\varepsilon$-difference bound for $Q_8$ near $\delta=-1$.
\end{itemize}
\item \emph{The hypothesis of Lemma~\ref{lem:env}.} A script checks $g>0$ and $y(n_1)\ge a/g$, with $a=5$, for every dominance and derivative certificate; all cases pass.
\end{enumerate}
The review found one error: the harmonic factor $\sum_{7\mid k\le N}1/k$ was missing from the $G_7$ arc-derivative estimate. It is included in the statements above, and all affected certificates were recomputed.


\begin{thebibliography}{99}
\bibitem{AB} G. E. Andrews and D. M. Bressoud, \emph{Vanishing coefficients in infinite product expansions}, J. Austral. Math. Soc. Ser. A \textbf{27} (1979), 199--202.
\bibitem{BHK} K. Bringmann, B. Heim and B. Kane, \emph{On a sign-change conjecture of Schlosser and Zhou}, J. Math. Anal. Appl. \textbf{548} (2025).
\bibitem{Chern} S. Chern, \emph{Asymptotics for the Fourier coefficients of eta-quotients}, J. Number Theory \textbf{199} (2019), 168--191.
\bibitem{HL} B. He and L. Li, \emph{Some conjectures of Schlosser and Zhou on sign patterns of the coefficients of infinite products}, arXiv:2509.10023.
\bibitem{Hirschhorn} M. Hirschhorn, \emph{On the expansion of a continued fraction of Gordon}, Ramanujan J. \textbf{5} (2001), 369--375.
\bibitem{Iseki} S. Iseki, \emph{A partition function with some congruence condition}, Amer. J. Math. \textbf{81} (1959), 939--961.
\bibitem{RS} B. Richmond and G. Szekeres, \emph{The Taylor coefficients of certain infinite products}, Acta Sci. Math. (Szeged) \textbf{40} (1978), 347--369.
\bibitem{Rad} H. Rademacher, \emph{Topics in Analytic Number Theory}, Grundlehren Math. Wiss. 169, Springer, 1973.
\bibitem{SZ} M. J. Schlosser and N. H. Zhou, \emph{On the infinite Borwein product raised to a positive real power}, Ramanujan J. \textbf{61} (2023), 515--543, doi:10.1007/s11139-021-00519-3; arXiv:2011.10552.
\bibitem{W} L. Wang, \emph{Sign changes of coefficients of powers of the infinite Borwein product}, Adv. in Appl. Math. \textbf{141} (2022), 102405.
\bibitem{arb} F. Johansson, \emph{Arb: efficient arbitrary-precision midpoint-radius interval arithmetic}, IEEE Trans. Comput. \textbf{66} (2017), 1281--1292.
\bibitem{flint} The FLINT team, \emph{FLINT: Fast Library for Number Theory}, version 3.6.0, \url{https://flintlib.org}.
\end{thebibliography}
\end{document}